\documentclass[reqno]{amsart}
\usepackage{a4wide,amssymb}
\usepackage{hyperref}

\AtBeginDocument{{\noindent\small
\emph{Electronic Journal of Differential Equations},
Vol. 2025 (2025), No. 55, pp. 1--21.\newline
ISSN: 1072-6691. URL: https://ejde.math.txstate.edu, https://ejde.math.unt.edu \newline
DOI: 10.58997/ejde.2025.55}
\thanks{\copyright 2025. This work is licensed under a CC BY 4.0 license.}
\vspace{8mm}}

\begin{document}
\title[\hfilneg EJDE-2025/55\hfil 
Inhomogeneous nonlinear Schr\"odinger equation]
{Non global solutions for non-radial inhomogeneous nonlinear 
Schr\"odinger equations}

\author[R. Bai, T. Saanouni \hfil EJDE-2025/55\hfilneg]
{Ruobing Bai, Tarek Saanouni}

\address{Ruobing Bai \newline
School of Mathematics and Statistics, 
Henan University, Kaifeng 475004, China}
\email{baimaths@hotmail.com}

\address{Tarek Saanouni \newline
Department of Mathematics, 
College of Science, 
Qassim University, Buraydah, Saudi Arabia}
\email{t.saanouni@qu.edu.sa}

\thanks{Submitted March 7, 2025. Published May 26, 2025.}
\subjclass[2020]{35Q55}
\keywords{Inhomogeneous Schr\"odinger problem;
nonlinear equations; bi-harmonic; 
\hfill\break\indent inverse square potential; finite time blow-up}

\begin{abstract}
This work concerns the inhomogeneous Schr\"odinger equation 
$$
 \mathrm{i}\partial_t u-\mathcal{K}_{s,\lambda}u +F(x,u)=0 , \quad 
 u(t,x):\mathbb{R}\times\mathbb{R}^N\to\mathbb{C}.
 $$
Here, $s\in\{1,2\}$, $N>2s$ and $\lambda>-(N-2)^2/4$. 
The linear Schr\"odinger operator is
$\mathcal{K}_{s,\lambda}:= (-\Delta)^s +(2-s)\frac{\lambda}{|x|^2}$,
and the focusing source term can be  local or non-local
$$
F(x,u)\in\{|x|^{-2\tau}|u|^{2(q-1)}u,|x|^{-\tau}|u|^{p-2}
\big(J_\alpha *|\cdot|^{-\tau}|u|^p\big)u\}.
$$
The Riesz potential is $J_\alpha(x)=C_{N,\alpha}|x|^{-(N-\alpha)}$, 
for certain $0<\alpha<N$. The singular decaying term $|x|^{-2\tau}$, 
for some $\tau>0$ gives an inhomogeneous non-linearity. 
One considers the inter-critical regime, namely 
$1+\frac{2(s-\tau)}N<q<1+\frac{2(s-\tau)}{N-2s}$ and 
$1+\frac{2(s-\tau)+\alpha}{N}<p<1+\frac{2(s-\tau)+\alpha}{N-2s}$. 
The purpose is to prove the finite time blow-up of solutions with datum 
in the energy space, not necessarily radial or with finite variance. 
The assumption on the data is expressed in two different ways. 
The first one is in the spirit of the potential well method due to 
Payne-Sattinger. The second one is the ground state threshold standard 
condition. The proof is based on Morawetz estimates and a non-global 
ordinary differential inequality. 
This work complements the recent paper by  Bai and  Li \cite{bl}
in many directions. 
\end{abstract}

\maketitle
\numberwithin{equation}{section}
\newtheorem{theorem}{Theorem}[section]
\newtheorem{lemma}[theorem]{Lemma}
\newtheorem{proposition}[theorem]{Proposition}
\allowdisplaybreaks

\section{Introduction}

This article concerns the Cauchy problem for an inhomogeneous generalized 
Hartree equation 
\begin{equation}
\begin{gathered}
\mathrm{i}\partial_t u-\mathcal{K}_{s,\lambda}u
+ |x|^{-\tau}|u|^{p-2}\big(J_\alpha *|\cdot|^{-\tau}|u|^p\big)u=0 ;\\
u(0,\cdot)=u_0,
\end{gathered} \label{S} 
\end{equation}
and the Cauchy problem for an inhomogeneous Schr\"odinger equation
\begin{equation}
\begin{gathered}
\mathrm{i}\partial_t u-\mathcal{K}_{s,\lambda}u
 + |x|^{-2\tau}|u|^{2(q-1)}u=0 ;\\
u(0,\cdot)=u_0. 
\end{gathered} \label{S'}
\end{equation}

Hereafter, $\frac{N}{2}>s\in\{1, 2\}$ and 
$u=u(t, x): \mathbb{R}\times\mathbb{R}^N \to \mathbb{C}$. The linear Schr\"odinger operator is 
denoted by $\mathcal{K}_{s,\lambda}:= (-\Delta)^s +(2-s)\frac{\lambda}{|x|^2}$. 
We considered 2 cases: 
The first one is 
$\mathcal K_{\lambda}:=\mathcal{K}_{1,\lambda}
= -\Delta +\frac{\lambda}{|x|^2}$, which corresponds to 
Schr\"odinger equation with inverse square potential. 
The second one is $\mathcal{K}_{2,\lambda}:= \Delta^2$, 
which corresponds to fourth-order Schr\"odinger equation. 
The inhomogeneous singular decaying term is 
$|\cdot|^{-2\tau}$ for some $\tau>0$. The Riesz-potential is defined on 
$\mathbb{R}^N$ by 
$$
J_\alpha:=\frac{\Gamma(\frac{N-\alpha}2)}{\Gamma(\frac\alpha2)
\pi^{N/2} 2^\alpha}\,|\cdot|^{\alpha-N},\quad  0<\alpha<N.
$$

In all this expression, one assumes that
\begin{equation}\label{cnd}
\min\{\tau,\alpha,N-\alpha,N-\tau,2-2\tau+\alpha\}>0.
\end{equation}

Motivated by the sharp Hardy inequality \cite{abde}, 
\begin{equation}\label{prt}
\frac{(N-2)^2}4\int_{\mathbb{R}^N}\frac{|f(x)|^2}{|x|^2}\,dx
\leq \int_{\mathbb{R}^N}|\nabla f(x)|^2\,dx,
\end{equation}
one assumes that $\lambda>-(N-2)^2/4$, which guarantees that extension of 
$-\Delta+\frac\lambda{|x|^2}$, denoted by $\mathcal K_{\lambda}$ 
is a positive operator. In the range 
$-\frac{(N-2)^2}4 <\lambda< 1-\frac{(N-2)^2}4$, the extension is not
 unique \cite{ksww,ect}. In such a case,
one picks the Friedrichs extension \cite{ksww,pst}.

Note that by the definition of the operator $\mathcal K_{\lambda}$ and 
Hardy estimate \eqref{prt}, one has
\begin{align}\label{norm}
\|\sqrt{\mathcal K_{\lambda}}\cdot\|=\Big(\|\nabla\cdot\|^2
+\lambda\|\frac\cdot{|x|}\|^2\Big)^{1/2}\simeq \|\cdot\|_{\dot H^1}.
\end{align}

The nonlinear equations of Schr\"odinger type \eqref{S} and \eqref{S'} 
model many physical phenomena. For $s=1$, they are used in nonlinear optical 
systems with spatially dependent interactions \cite{bgvt}. 
In particular, when $\lambda = 0$, they can be thought of as modeling 
inhomogeneities in the medium in which the
wave propagates \cite{kmvbt}. When $\tau = 0$, they model a quantum field 
equations or black hole solutions of the Einstein’s equations \cite{ksww}. 
For $s=2$, the above equations are called fourth-order Schr\"odinger 
equations. The bi-harmonic Schr\"odinger problem was considered first 
in \cite{kar,kars} to take into account
the role of small fourth-order dispersion terms in the propagation of 
intense laser beams in
a bulk medium with a Kerr non-linearity. The source term can be understood 
as a nonlinear potential affected by electron density \cite{lb}. 
 
The literature dealing with \eqref{S} and \eqref{S'} is copious, and 
naturally some references are missing here. Let us start with the 
Schr\"odinger equation with inverse square potential, which corresponds 
to $s=1$. Using the energy method,  \cite{tsuz,tsuz1} investigated the 
local well-posedness in the energy space. Moreover, the local solution 
extends globally in time if either defocusing case or focusing, 
mass-subcritical case. Later on, \cite{cg} revisits the same problem, 
where the authors studied the local well-posedness and small data global 
well-posedness in the energy-sub-critical case by using the standard
Strichartz estimates combined with the fixed point argument. 
See also \cite{ajk} for the ground state threshold of global existence 
versus blow-up dichotomy in the inter-critical regime. 
Furthermore, \cite{cg} showed a scattering criterion and constructed a 
wave operator for the inter-critical case. The well-posedness and blow-up 
in the energy critical regime were investigated in \cite{jak}. 
The inhomogeneous generalized Hartree equation was treated first by the 
author \cite{mt}, where the ground state threshold dichotomy was investigated 
using a sharp adapted Gagliargo-Nirenberg type estimate. After that, the 
second author treated the intermediate case in the sense that \eqref{S} is 
locally well-posed in $\dot H^1\cap\dot H^{s_c}$, $0 < s_c < 1$, but
this does not imply the inter-critical case $H^{s_c}$. 
The scattering under the ground state threshold with spherically symmetric 
data, was proved by the second author \cite{sx}. The scattering was extended 
to the non-radial regime in \cite{cx}. The well-posedness in the 
energy-critical regime was investigated recently \cite{kls,sk}. 
To this end, the authors approach to the matter based on the Sobolev-Lorentz 
space which can lead to perform a finer analysis. 
This is because it makes it possible to control the non-linearity involving 
the singularity $|x|^{-\tau}$ as well as the Riesz potential more effectively.
Now, one deals with the bi-harmonic case, namely $s=2$. For a local source 
term, in \cite{gp}, the local well-posedness was obtained in the energy 
sub-critical regime. This result was improved in \cite{ajk1}. 
The scattering was investigated in \cite{gp1,cg1,vdd3}. For a non-local 
source term, the local existence of energy solutions and the scattering were 
proved by the second author in \cite{st4,sg}. See also \cite{sg2,sp2} 
for the energy-critical regime.

The finite time concentration of energy solutions to non-linear 
Schr\"odinger equations has a long history. Indeed, in the 
mass-super-critical focusing regime, it is known that an energy data with 
finite variance or which is radial gives a blowing up solution for negative 
energy \cite{glas,ot1}. A similar result for non-radial data and with 
infinite variance is open except for $N=1$, see \cite{ot2}. 
The results of blow-up in some other situations can be referred to 
\cite{DWZ, HR, YM, FM} and references therein.
Recently, some works try to remove the radial or finite variance data 
assumption in the inhomogeneous case. Indeed, the second author proved 
in \cite{bl} the finite time blow-up of energy solutions under the ground 
state threshold in a restricted range of the source term exponent. 
In the mass-critical regime, the blow-up of energy solutions with negative 
energy was obtained recently \cite{cf}.

The blow-up of energy solutions to bi-harmonic Schr\"odinger equations was 
open for a long time because of the lack of a variance identity.
 Many authors investigated the blow-up of radial solutions, since the 
pioneering work \cite{bhl} using a localized virial identity for radial datum.
See, for instance \cite{cow,sg}. Recently the blow-up for arbitrary datum
 with negative energy, in the energy space, was obtained in \cite{vdd2} 
for a perturbed bi-harmonic NLS. This result don't extend to \eqref{S'} 
for $s=2$.

The purpose of this article is to investigate the finite time blow-up of 
energy solutions to the Schr\"odinger problems \eqref{S} and \eqref{S'}. 
The novelty is to prove the non-global existence of solutions with arbitrary 
negative energy datum. Precisely, one don't require any radial or finite 
variance assumption for the datum. In the case $s=1$, this work complements 
the paper of the second author \cite{bl} for $\lambda\neq0$ and for a 
non-local source term. Moreover, one considers a weaker assumption on the 
datum. In the case $s=2$, this work complements the paper \cite{vdd2} 
to the inter-critical regime, and for a non-local source term. 
Furthermore, this work gives a natural complement of the paper \cite{st4},
where the first author deals with the scattering of the bi-harmonic
Schr\"odinger equation in the inter-critical focusing regime under the 
ground state threshold.

The rest of this article is organized as follows. 
The next section contains the main results and some useful estimates. 
Sections 3 and 4 contain the proofs of the main results. 

\section{Background and main results}

This section contains the main results and some useful estimates.

\subsection{Preliminaries}
Here and hereafter, one denotes for simplicity some standard Lebesgue 
and Sobolev spaces and norms as follows
\[
L^r:=L^r(\mathbb{R}^N),\quad W^{s,r}:=W^{s,r}(\mathbb{R}^N),\quad 
H^s:=W^{s,2}, \quad 
\|\cdot\|_r:=\|\cdot\|_{L^r},\quad
\|\cdot\|:=\|\cdot\|_2.
\]
Let us also define the real numbers
\[
B:=\frac{Np-N-\alpha+2\tau}{s},\quad A:=2p-B, \quad
B':=\frac{Nq-N+2\tau}{s},\quad A':=2q-B'.
\]
If $u\in {H^s}$, one defines the quantities related to energy solutions 
of \eqref{S} and \eqref{S'},
\begin{gather*}
\mathcal P[u]:=\int_{\mathbb{R}^N} |x|^{-\tau}\big(J_\alpha 
 *|\cdot|^{-\tau}|u|^p\big)|u|^p\,dx,\quad 
 \mathcal Q[u]:=\int_{\mathbb{R}^N} |x|^{-2\tau}|u|^{2q}\,dx,\\
\mathcal I[u]:=\|\sqrt{\mathcal K_{s,\lambda}} u\|^2
 -\frac{B}{2p}\mathcal P[u],\quad
  \mathcal J[u]:=\|\sqrt{\mathcal K_{s,\lambda}}
   u\|^2-\frac{B'}{2q}\mathcal Q[u];\\
\mathcal M[u]:= \int_{\mathbb{R}^N}|u(x)|^2\,dx ,\quad 
\mathcal E[{u}]:=\|\sqrt{\mathcal K_{s,\lambda}} u\|^2
-\frac1p\mathcal P[u],\\
 \mathcal E'[{u}]:=\|\sqrt{\mathcal K_{s,\lambda}} u\|^2-\frac1q\mathcal Q[u].
\end{gather*}
We denote also the so-called actions
\begin{gather}
   \mathcal S[u]:=\mathcal E[{u}]+\mathcal M[{u}],\label{su}\\
 \mathcal S'[u]:=\mathcal E'[{u}]+\mathcal M[{u}].\label{s'u}
\end{gather}
Take also the real numbers
\begin{gather}
m:=\inf_{0\neq u\in H^s}\big\{\mathcal S[u]: 
 \mathcal{I}[u]=0\big\};\label{m}\\
m':=\inf_{0\neq u\in H^s}\big\{\mathcal S'[u]: 
 \mathcal{J}[u]=0\big\}.\label{m'}
\end{gather}
Finally, we define the sets, which are non-empty with a scaling argument
\begin{gather}
\mathcal{A}^-:=\big\{ u\in H^s: 
 \mathcal{S}[u]<m: \mathcal{I}[u]<0\big\}\label{a-},\\
\mathcal{A'}^-:=\big\{ u\in H^s:
 \mathcal{S'}[u]<m',\quad \mathcal{J}[u]<0\big\}\label{a'-}.
\end{gather}
Then equation \eqref{S} has the scaling invariance
\begin{equation}\label{scal}
{u}_\kappa:=\kappa^\frac{2s-2\tau+\alpha}{2(p-1)}u(\kappa^{2s}
\cdot,\kappa\cdot),\quad\kappa>0.
\end{equation}
The critical exponent $s_c$ keeps invariant the homogeneous
Sobolev norm
$$
\|{u}_\kappa(t)\|_{\dot H^\mu}=\kappa^{\mu-(\frac N2
-\frac{2s-2\tau+\alpha}{2(p-1)})}\|{u}(\kappa^{2s} t)\|_{\dot H^\mu}
:=\kappa^{\mu-s_c}\|{u}(\kappa^{2s} t)\|_{\dot H^\mu}.
$$
Two cases are of particular interest in the physical context. 
The first one $s_c=0$ corresponds to the mass-critical case which is 
equivalent to $p=p_c:=1+\frac{2s-2\tau+\alpha}N$. 
This case is related to the conservation of the mass $\mathcal{M}$ 
given above. The second one is the energy-critical case $s_c=s$, 
which corresponds to $p=p^c:=1+\frac{2s-2\tau+\alpha}{N-2s}$. 
This case is related to the conservation of the energy $\mathcal E$ 
defined above. A particular periodic global solution of \eqref{S} 
takes the form $e^{\mathrm{i}t}\varphi$, where $\varphi$ satisfies
\begin{equation}\label{E}
\mathcal K_{s,\lambda}\varphi+\varphi=|x|^{-\tau}|\varphi|^{p-2}
\big(J_\alpha *|\cdot|^{-\tau}|\varphi|^p\big)\varphi,\quad 
0\neq \varphi\in {H^s}.
\end{equation}
The equation \eqref{S'} has the scaling invariance
\begin{equation}\label{scal'}
{u}_\kappa:=\kappa^\frac{s-\tau}{q-1}u(\kappa^{2s}\cdot,
\kappa\cdot),\quad\kappa>0.
\end{equation}
The critical exponent $s_c'$ keeps invariant the following homogeneous 
Sobolev norm
$$
\|{u}_\kappa(t)\|_{\dot H^\mu}
=\kappa^{\mu-(\frac N2-\frac{s-\tau}{q-1})}\|{u}(\kappa^{2s} t)
\|_{\dot H^\mu}:=\kappa^{\mu-s_c'}\|{u}(\kappa^{2s} t)\|_{\dot H^\mu}.
$$
Two cases are of particular interest in the physical context. 
The first one $s_c'=0$ corresponds to the mass-critical case which is 
equivalent to $q=q_c:=1+\frac{2s-2\tau}N$. 
This case is related to the conservation of the mass. The second one is 
the energy-critical case $s_c'=s$, which corresponds to 
$q=q^c:=1+\frac{2s-2\tau}{N-2s}$. This case is related to the 
conservation of the energy $ \mathcal E'$ defined above. A particular 
periodic global solution of \eqref{S'} takes the form $e^{\mathrm{i}t}\psi$, 
where $\psi$ satisfies
\begin{equation}\label{E'}
\mathcal K_{s,\lambda}\psi+\psi=|x|^{-2\tau}|\psi|^{2(q-1)}\psi,
\quad 0\neq \psi\in {H^s}.
\end{equation}

The existence of such a ground state is related to the next 
Gagliardo-Nirenberg type inequalities \cite{sg,stsubm}.

\begin{proposition}\label{gag}
Let $s\in\{1,2\}$, $N>2s$, $0<\alpha<N$ and $1+\frac\alpha N< p< p^c$. 
If $s=1$, one assumes that $\lambda>-\frac{(N-2)^2}{4}$ and \eqref{cnd} holds.
Moreover, if $s=2$, one assumes that 
$0<2\tau<\min\{N+\alpha,4(1+\frac\alpha N)\}$. Thus, 
\begin{enumerate}
\item There exists
a sharp constant $C_{N,p,\tau,\alpha,\lambda}>0$ such that for all 
$u\in H^s$,
\begin{equation}\label{ineq}
\int_{\mathbb{R}^N}|x|^{-\tau}|u|^p\big(J_\alpha*|\cdot|^{-\tau}|u|^p\big)\,dx
\leq C_{N,p,\tau,\alpha,\lambda}\|u\|^A\|\sqrt{\mathcal
 K_{s,\lambda}} u\|^B;
\end{equation}

\item there exists $\varphi$ a solution to \eqref{E} satisfying
\begin{equation}\label{part3}
C_{N,p,\tau,\alpha,\lambda}=\frac{2p}{A}\big(\frac AB\big)^{B/2}
\frac1{\|\varphi\|^{2(p-1)}};
\end{equation}

\item one has the following Pohozaev identities
\begin{equation}\label{poh}
\mathcal P[\varphi]=\frac{2p}{A}\mathcal M[\varphi]
=\frac{2p}{B}\|\sqrt{\mathcal K_{s,\lambda}}\varphi\|^2.
\end{equation}
\end{enumerate}
\end{proposition}

The next Gagliardo-Nirenberg type inequality \cite{sg,cg} is essential 
to estimate an eventual solution to the problem \eqref{S'}.

\begin{proposition}\label{gag2}
Let $s\in\{1,2\}$, $N>2s$, $\lambda>-\frac{(N-2)^2}{4}$, $0<\tau<s$ and 
$1< q< q^c$. Thus, 
\begin{enumerate}
\item
there exists a sharp constant $C_{N,q,\tau,\lambda}>0$, such that for all 
$v\in H^s$,
$$
\int_{\mathbb{R}^N}|x|^{-2\tau}|v|^{2q}\,dx\leq C_{N,q,\tau,\lambda}\|v\|^{A'}
\|\sqrt{\mathcal K_{s,\lambda}} v\|^{B'}
$$

\item
 there exists $\psi$ a solution to \eqref{E'} satisfying
\begin{equation}\label{part32}
C(N,q,\tau)=\frac{2q}{A'}(\frac{A'}{B'})^{\frac{B'}2}\|\psi\|^{-2(q-1)},
\end{equation}
where $\psi$ is a solution to \eqref{E'};

\item one has the following Pohozaev identities
\begin{equation}\label{poh'}
\mathcal Q[\psi]=\frac{2q}{A'}\mathcal M[\psi]
=\frac{2q}{B'}\|\sqrt{\mathcal K_{s,\lambda}} \psi\|^2.
\end{equation}
\end{enumerate}
\end{proposition}

In the inter-critical regime $0<s_c<{s}$, one denotes the positive real 
number $\frac{s}{s_c}-1:=\alpha_c\in(0,1)$, $\varphi$ be a ground state 
of \eqref{E} and the scale invariant quantities
\[
\mathcal{ME}[u_0]:=\Big(\frac{\mathcal M[u_0]}{\mathcal M[\varphi]}
 \Big)^{\alpha_c}\Big(\frac{\mathcal E[u_0]}{ \mathcal E[\varphi]}\Big),\quad
\mathcal{MG}[u_0]:=\Big(\frac{\|u_0\|}{\|\varphi\|}\Big)^{\alpha_c}
\Big(\frac{\|\sqrt{\mathcal K_{s,\lambda}} u_0\|}
{\|\sqrt{\mathcal K_{s,\lambda}} \varphi\|}\Big).
\]
Similarly, in the inter-critical regime $0<s_c'<{s}$, one denotes 
the positive real number $\frac{s}{s_c'}-1:=\alpha_c'\in(0,1)$, $\psi$ 
be a ground state of \eqref{E'} and the scale invariant quantities
\[
\mathcal{ME'}[u_0]:=\Big(\frac{\mathcal M[u_0]}{\mathcal M[\psi]}
 \Big)^{\alpha_c'}\Big(\frac{\mathcal E'[u_0]}{\mathcal E'[\psi]}\Big),\quad
\mathcal{MG'}[u_0]:=\Big(\frac{\|u_0\|}{\|\psi\|}\Big)^{\alpha_c'}
\Big(\frac{\|\sqrt{\mathcal K_{s,\lambda}} u_0\|}
{\|\sqrt{\mathcal K_{s,\lambda}} \psi\|}\Big).
\]
In the next sub-section, one lists the contributions of this note.

\subsection{Main results}

First, one deals with the non-global existence of energy solutions to the 
generalized Hartree problem \eqref{S}.

\begin{theorem}\label{t1}
Let $s\in\{1,2\}$, $N>2s$, $0<\alpha<N$, $\lambda\geq0$ and 
$0<\tau<{s\frac{\alpha+N}N}$ {such that \eqref{cnd} holds}. 
Suppose that $\max\{2,p_c\}< p <p^c$ and $p\leq 1+\frac{2s+\alpha{-\tau}}N$. 
Take $\varphi$ be a ground state solution to \eqref{E} and 
$u\in C_{T^*}(H^s)$ be a maximal solution of the focusing problem \eqref{S}. 
Thus, $u$ blows-up in finite time if one of the following assumptions holds
\begin{gather} \label{ss}
u_0\in \mathcal A^-, \\
\label{t13}
\mathcal{MG}[u_0]>1>\mathcal{ME}[u_0].
\end{gather}
\end{theorem}

In view of the results stated in the above theorem, some comments 
are in order.
\begin{itemize}
\item In \cite{st4}, the local existence of energy solutions for \eqref{S} 
with $s=2$ was proved under the supplementary assumption 
$0<2\tau<\min\{4(1+\frac\alpha N),-N+8+\alpha\}$.
 Moreover, in \cite{stsubm}, the local existence of energy solutions 
for \eqref{S} with $s=1$ was proved under the supplementary assumption 
$1+\alpha-2\tau>0$. 

\item  The space $ \mathcal A^-$ is proved to be stable under the flow of 
\eqref{S}.

\item  The first part of the Theorem follows the potential well theory 
due to Payne-Sattinger \cite{pyn}.

\item The assumption on the source term exponent, can be written as 
${2<B\leq 2+\frac\tau s}$.

\item The slab $(p_c,1+\frac{2s+\alpha{-\tau}}N]$ has a length of 
${\frac{\tau}{N}}$, which is independent of $s$.

\item The restriction $\lambda\geq0$ is needed in the proof.

\item The above result doesn't extend to the limiting case $\tau=0$, 
which is still an open problem. This gives an essential difference 
between the NLS and the INLS.

\item In a paper in progress, the authors treat the finite time blow-up 
of energy solutions in the mass-critical bi-harmonic regime. 
\end{itemize}

Second, one deals with the non-global existence of energy solutions to 
the Schr\"odinger problem \eqref{S'}.

\begin{theorem}\label{t1'}
Let $s\in\{1,2\}$, $N>2s$, $\lambda\geq0$ and $0<\tau<2$. Assume that 
$q_c< q<q^c$ and $q\leq1+\frac{2s+2\tau(s-1)}N$. 
Take $\psi$ be a ground state solution to \eqref{E'} and 
$u\in C_{T^*}(H^s)$ be a maximal solution of the focusing problem \eqref{S'}. 
Thus, $u$ blows-up in finite time if one of the following assumptions holds
\begin{gather} \label{ss'}
u_0\in \mathcal A'^-, \\
\label{t13'}
\mathcal{MG'}[u_0]>1>\mathcal{ME'}[u_0].
\end{gather}
\end{theorem}

In view of the  results stated in the above theorem, 
some comments are in order.
\begin{itemize}
\item In \cite{cg,tsuz}, the local existence of energy solutions to 
\eqref{S'} for $s=1$ was proved under the supplementary assumption 
$\tau<1$. In \cite{cg,tsuz,gp}, the local existence of energy solutions 
to \eqref{S'} for $s=2$ was proved under the supplementary assumption 
$q>1+\frac{1-2\tau}{N}$.

\item Assumption \eqref{t13'} is used to prove that \eqref{ss'} is stable 
under the flow of \eqref{S'}.

\item In \cite{bl}, the first author proved the finite time blow-up of 
energy solutions for \eqref{S'} for $s=1$ and $\lambda=0$ under the 
assumption \eqref{t13'}. 
\end{itemize}

\subsection{Useful estimates}
In this sub-section, one gives some standard tools needed in the sequel. 
Let us start with Hardy-Littlewood-Sobolev inequality \cite{el}.

\begin{lemma}\label{hls}
Let $N\geq1$ and $0 <\alpha < N$.
\begin{enumerate}
\item
Let $r>1$ such that $\frac1r=\frac1s+\frac\alpha N$. Then,
$$
\|J_\alpha*g\|_s\leq C_{N,s,\alpha}\|g\|_{r},\quad\forall g\in L^r.
$$

\item Let $1<s,r<\infty$ be such that 
$\frac1r +\frac1s=\frac1t +\frac\alpha N$. Then
$$
\|f(J_\alpha*g)\|_t\leq C_{N,s,\alpha}\|f\|_{r}\|g\|_{s},\quad
\forall (f,g)\in L^r\times L^s.
$$
\end{enumerate}
\end{lemma}

Let $\xi:\mathbb{R}^N\to\mathbb{R}$ be a convex smooth function. We define the variance 
potential
\begin{equation}\label{vrl}
V_\xi:=\int_{\mathbb{R}^N}\xi(x)|u(\cdot,x)|^2\,dx,
\end{equation}
and the Morawetz action
\begin{equation}\label{mrwtz}
M_\xi=2\Im\int_{\mathbb{R}^N} \bar u(\nabla\xi\cdot\nabla u)\,dx
:=2\Im\int_{\mathbb{R}^N} \bar u(\xi_ju_j)\,dx,
\end{equation}
{where here and sequel, repeated indices are summed.}
 Let us give a Morawetz type estimate for the Schr\"odinger equation 
 with inverse square potential \cite{ajk}.
 
\begin{proposition}\label{mrwz1}
Take $u, v\in C_{T^*}(H^1)$ be the local solutions to \eqref{S} and 
\eqref{S'} for $s=1$, respectively. Let $\xi:\mathbb{R}^N\to\mathbb{R}$ be a smooth function. 
Then, the following equality holds on $[0,T^*)$,
\begin{align*}
&V_\xi''[u]=M_\xi'[u]\\
&=4\int_{\mathbb{R}^N}\partial_l\partial_k\xi\Re(\partial_ku\partial_l\bar u)\,dx
 -\int_{\mathbb{R}^N}\Delta^2\xi|u|^2\,dx+4\lambda\int_{\mathbb{R}^N}\nabla\xi
  \cdot x\frac{|u|^2}{|x|^4}\,dx\\
&\quad +2(\frac2p-1)\int_{\mathbb{R}^N}\Delta\xi|x|^{-\tau}|u|^p(J_\alpha*|
 \cdot|^{-\tau}|u|^p)\,dx
 +\frac{4}p\int_{\mathbb{R}^N}\nabla\xi\cdot\nabla(|x|^{-\tau})|u|^p
  \big(J_\alpha*|\cdot|^{-\tau}|u|^p\big)\,dx\\
&\quad+ \frac{4}p(\alpha-N)\int_{\mathbb{R}^N}|x|^{-\tau}|u|^p\nabla
 \xi(\frac\cdot{|\cdot|^2}J_\alpha*|\cdot|^{-\tau}|u|^p)\,dx.
\end{align*}
Moreover,
\begin{align*}
V_\xi''[v]=M_\xi'[v]
&= 4\int_{\mathbb{R}^N}\partial_l\partial_k\xi\Re(\partial_kv\partial_l\bar v)\,dx
 -\int_{\mathbb{R}^N}\Delta^2\xi|v|^2\,dx+4\lambda\int_{\mathbb{R}^N}\nabla\xi
  \cdot x\frac{|v|^2}{|x|^4}\,dx\\
&\quad +2(\frac1q-1)\int_{\mathbb{R}^N}\Delta\xi|x|^{-2\tau}|v|^{2q}\,dx
 +\frac{2}q\int_{\mathbb{R}^N}\nabla\xi\cdot\nabla(|x|^{-2\tau})|v|^{2q}\,dx.
\end{align*}
\end{proposition}

Finally, one gives a Morawetz estimate for the bi-harmonic Schr\"odinger 
equation \cite{st4}.

\begin{proposition}\label{mrwz}
Take $u,v\in C_{T^*}(H^2)$ be the local solutions to \eqref{S} and 
\eqref{S'} for $s=2$, respectively. Let $\xi:\mathbb{R}^N\to\mathbb{R}$ be a smooth 
function. Then, the following equalities hold on $[0,T^*)$,
\begin{gather} \label{mrw1}
\begin{aligned}
M_\xi'[u]
&=-2\int_{\mathbb{R}^N}\Big(2\partial_{jk}\Delta\xi\partial_ju\partial_k
 \bar u-\frac12(\Delta^3\xi)|u|^2-4\partial_{jk}
 \xi\partial_{ik}u\partial_{ij}\bar u+\Delta^2\xi|\nabla u|^2\Big)
 \,dx \\
&\quad -2\Big((1-\frac2p)\int_{\mathbb{R}^N}\Delta\xi(J_\alpha*|
\cdot|^{-\tau}|u|^p)|x|^{-\tau}|u|^p\,dx \\
&\quad -\frac2{p}\int_{\mathbb{R}^N}\partial_k\xi\partial_k(|x|^{-\tau}[J_\alpha*|
 \cdot|^{-\tau}|u|^p])|u|^p\,dx\Big),
\end{aligned}\\
\begin{aligned}
M_\xi'[v]
&=-2\int_{\mathbb{R}^N}\Big(2\partial_{jk}\Delta\xi\partial_jv\partial_k\bar v
 -\frac12(\Delta^3\xi)|v|^2-4\partial_{jk}\xi\partial_{ik}v\partial_{ij}
  \bar v \\
&\quad +\Delta^2\xi|\nabla v|^2+\frac{q-1}{q}(\Delta\xi)|x|^{-2\tau}|v|^{2q}
 -\frac1q\nabla\xi\cdot\nabla(|x|^{-2\tau})|v|^{2q}\Big)\,dx
\end{aligned}\label{mrw2}.
\end{gather}
\end{proposition}

The next radial identities will be useful in the sequel.
\begin{gather}
\frac{\partial^2}{\partial x_l\partial x_k}
 :=\partial_l\partial_k=\Big(\frac{\delta_{lk}}r
  -\frac{x_lx_k}{r^3}\Big)\partial_r
  +\frac{x_lx_k}{r^2}\partial_r^2, \label{symm}\\
\Delta=\partial_r^2+\frac{N-1}r\partial_r, \label{sym1}\\
\nabla=\frac x{r}\partial_r\label{sym2}.
\end{gather}
In the rest of this note, one takes a smooth radial function 
$\xi(x):=\xi(|x|)$ such that 
$$
\xi:r\to\begin{cases}
r^2,&\text{if } 0\leq r\leq1;\\
0,  &\text{if }  r\geq10.
\end{cases}
$$
So, on the unit ball of $\mathbb{R}^N$, one has
$$
\xi_{ij} =2\delta_{ij},\quad
\Delta\xi =2N,\quad \partial^\gamma\xi= 0 \quad\text{for }
 |\gamma| \geq 3.
$$
Now, for $R>0$,  via \eqref{mrwtz}, one takes
$$
\xi_R:=R^2\xi(\frac{|\cdot|}R)\quad\text{and}\quad M_R:=M_{\xi_R}.
$$
By \cite[Lemma 2.1]{vdd2}, one can impose that
\begin{align}\label{prpr}
\max\{\frac{\xi_R'}r-2, \xi_R''-\frac{\xi_R'}r\}\leq 0.
\end{align}
From now one hides the time variable $t$ for simplicity, 
displaying it out only when necessary.
Moreover, one denotes the centered ball of $\mathbb{R}^N$ with radius $R>0$ 
and its complementary, respectively by $B(R)$ and $B^c(R)$.
In what follows, one proves the main results of this note.

\section{Schr\"odinger equation with non-local source term}

In this section, we establish Theorem \ref{t1}.

\subsection{Schr\"odinger equation with inverse square potential}
In this sub-section, one takes $s=1$. 

\subsubsection*{First case}
Assume that \eqref{ss} holds. We start with the next auxiliary result.

\begin{lemma}\label{ky}
\begin{enumerate}
\item The set $\mathcal{A}^-$ is stable under the flow of \eqref{S}.
\item There exists $\varepsilon>0$, such that for any $t\in [0, T^*)$,
\begin{equation}
 \mathcal I[u(t)]+\varepsilon \|\sqrt{\mathcal K_\lambda}u(t)\|^2
 \leq-\frac B4\big(m-\mathcal{S}[u(t)]\big).\label{nw}
\end{equation}
\end{enumerate}
\end{lemma}

\begin{proof}
(1) Assume that $u_0\in\mathcal{A}^-$ and that there is $0<t_0<T^*$ such that
$u(t_0)\notin\mathcal{A}^-$. This implies that $\mathcal{I}[u(t_0)]\geq0$ 
and by a continuity argument, there is $0<t_1$ such that 
$\mathcal{I}[u(t_1)]=0$ and $\mathcal S[u(t_1)]<m$. 
This contradicts the definition of $m$ and proves the first point of 
Lemma \ref{ky}.
 
(2) Now, taking the scaling $u_\rho:=\rho^{\frac N2}u(\rho\cdot)$ for 
$\rho>0$, we compute 
\begin{gather}
\|u_\rho\|=\|u\|;\label{scl1}\\
\|\sqrt{\mathcal K_\lambda}u_\rho\|
= \rho\|\sqrt{\mathcal K_\lambda}u\|;\label{scl2}\\
\mathcal P[u_\rho]= \rho^B\mathcal P[u]\label{scl4}.
\end{gather}
Moreover, take the real function $\digamma:\rho\mapsto \mathcal{S}[u_\rho]$, 
we obtain $\digamma(\rho)=\rho^2\|\sqrt{\mathcal K_\lambda}u\|^2+\|u\|^2
-\frac{\rho^B}{p}\mathcal P[u]$ and the first derivative reads
\begin{equation}
 \digamma'(\rho)=2\rho\|\sqrt{\mathcal K_\lambda}u\|^2
-B\frac{\rho^{B-1}}{p}\mathcal P[u]=2\rho^{-1}\mathcal I[u_\rho]\label{rv1}.
\end{equation}
Hence, this implies
\begin{equation}
 \rho\digamma'(\rho)=2\rho^2\|\sqrt{\mathcal K_\lambda}u\|^2
 -B\frac{\rho^{B}}{p}\mathcal P[u]=2\mathcal I[u_\rho].\label{rv2}
\end{equation}
Moreover, since $B>2$, we obtain
\begin{equation}
\big(\rho\digamma'(\rho)\big)'
 =4\rho\|\sqrt{\mathcal K_\lambda}u\|^2-B^2\frac{\rho^{B-1}}{p}\mathcal P[u] 
 =B\digamma'(\rho)-2(B-2)\rho\|\sqrt{\mathcal K_\lambda}u\|^2
 \leq B\digamma'(\rho).\label{rv3}
\end{equation}
Now, we claim that there exists $\rho_0\in(0,1)$ such that
\begin{equation}
 \mathcal I[u_{\rho_0}]=0.\label{claim}
\end{equation}
Indeed, by \eqref{scl2} and \eqref{scl4}, we have
\begin{equation}
 \mathcal I[u_{\rho}]
=\rho^2\Big(\|\sqrt{\mathcal K_\lambda}u\|^2-\frac{\rho^{B-2}}{2p}
 \mathcal P[u]\Big) 
:=\rho^2\aleph(\rho).\label{rv4}
\end{equation}
Note that $\aleph(0)>0$ and $\aleph(1)=\mathcal I[u]<0$, thus there 
exists $\rho_0\in(0,1)$ such that $\aleph(\rho_0)=0$, the claim is proved.
Thus, by \eqref{rv1}, we have $\digamma'(\rho_0)=0$ and 
$\digamma(\rho_0)=\mathcal S[u_{\rho_0}]\geq m$. 
Hence, an integration of \eqref{rv3} on $[\rho_0,1]$ gives
\[
 \digamma'(1)-\rho_0\digamma'(\rho_0)\leq B  \digamma(1)
 -B  \digamma(\rho_0).
\]
Note that $ \digamma'(1)=2\mathcal I[u]$, 
$\rho_0\digamma'(\rho_0)=2\mathcal I[u_{\rho_0}]=0$, and 
$ \digamma(1)=\mathcal S[u]$, the above inequality further implies 
\begin{align}
 \mathcal{I}[u] \leq-\frac B2\big(m-\mathcal{S}[u]\big).\label{rv5}
\end{align}
On the other hand, we write
\begin{equation}
 \|\sqrt{\mathcal K_\lambda}u\|^2
 =\frac{B}{B-2}\big(\mathcal{S}[u]-\frac2B\mathcal{I}[u]-\|u\|^2\big).
\end{equation}
Hence, by \eqref{rv5}, there exists $0<\varepsilon\ll1$, such that
\begin{equation}
\begin{aligned}
\mathcal{I}[u]+\varepsilon\|\sqrt{\mathcal K_\lambda}u\|^2
&=\Big(1-\frac{2\varepsilon}{B-2}\Big)\mathcal{I}[u]
 +\varepsilon\frac{B}{B-2} \big(\mathcal{S}[u]-\|u\|^2\big) \\
&\leq-\frac B2\Big(1-\frac{2\varepsilon}{B-2}\Big)
 \big(m-\mathcal{S}[u]\big)+\varepsilon\frac{B}{B-2} \mathcal{S}[u] \\
&\leq-\frac B4\big(m-\mathcal{S}[u]\big)
\end{aligned}\label{4.251}.
\end{equation}
The last statement of Lemma \ref{ky} is proved by \eqref{4.251}.
\end{proof}

Now we turn to the proof of the main results. 
Taking into account Proposition \ref{mrwz1}, one has $M_R':=(L)+(N)$, where
\[
(L)=-\int_{\mathbb{R}^N}\Delta^2\xi_R|u|^2\,dx
 +4\int_{\mathbb{R}^N}\partial_l\partial_k\xi_R\Re(\partial_ku\partial_l\bar u)\,dx
 +4\lambda\int_{\mathbb{R}^N}\nabla\xi_R\cdot x\frac{|u|^2}{|x|^4}\,dx,
\]
and
\begin{align*}
(N)
&=2(\frac2p-1)\int_{\mathbb{R}^N}\Delta\xi_R|x|^{-\tau}|u|^p(J_\alpha*
 |\cdot|^{-\tau}|u|^p)\,dx\\
&\quad -\frac{4\tau}p\int_{\mathbb{R}^N}x\cdot\nabla\xi_R|x|^{-\tau-2}|u|^p
 \big(J_\alpha*|\cdot|^{-\tau}|u|^p\big)\,dx\\
&\quad +\frac{4}p(\alpha-N)\int_{\mathbb{R}^N}|x|^{-\tau}|u|^p\nabla\xi_R
 \Big(\frac\cdot{|\cdot|^2}J_\alpha*|\cdot|^{-\tau}|u|^p\Big)\,dx\\
&:=(N)_1+(N)_2+(N)_3.
\end{align*}
For the term $(L)$, {with the properties of $\xi_R$, namely 
\eqref{prpr} and the radial identities}, it follows that
\begin{equation} \label{36461}
\begin{aligned}
(L)&= -\int_{\mathbb{R}^N}\Delta^2\xi_R|u|^2\,dx
 +4\int_{\mathbb{R}^N}|\nabla u|^2\frac{\xi_R'}r\,dx
 +4\int_{\mathbb{R}^N}|x\cdot\nabla u|^2(\frac{\xi_R''}{r^2}
 -\frac{\xi_R'}{r^3})\,dx \\
&\quad +4\lambda\int_{\mathbb{R}^N}\frac{|u|^2}{r^3}\xi_R'\,dx \\
&\leq -\int_{\mathbb{R}^N}\Delta^2\xi_R|u|^2\,dx+8\int_{\mathbb{R}^N}|\nabla u|^2\,dx
 +8\lambda\int_{\mathbb{R}^N}\frac{|u|^2}{r^2}\,dx.
\end{aligned}
\end{equation}
For the terms $(N)_1$ and $(N)_2$ in $(N)$, taking account of the 
truncation function properties and the conservation laws, one has
\[
(N)_1+(N)_2=2(\frac {4N-4\tau}p-2N)\mathcal P[u]
+O\Big(\int_{B^c(R)}|x|^{-\tau}|u|^p\big(J_\alpha*
 |\cdot|^{-\tau}|u|^p\big)\,dx\Big).
\]
For the third term $(N)_3$ in $(N)$, with the calculations done in \cite{st4}, 
one has
\begin{equation} \label{372}
\begin{aligned}
(N)_3
&= \frac{4(\alpha-N)}{p}\int_{B(R)\times B(R)}J_\alpha
 (x-y)|y|^{-\tau}|u(y)|^p|x|^{-\tau}|u(x)|^p\,dx\,dy \\
&\quad +O\Big(\int_{B^c(R)}\big(J_\alpha*|\cdot|^{-\tau}|u|^p\big)
 |x|^{-\tau}|u|^p\,dx\Big) \\
&= \frac{4(\alpha-N)}{p}\int_{B(R)}\big(J_\alpha*
 |\cdot|^{-\tau}|u|^p\big)|x|^{-\tau}|u(x)|^p\,dx \\
&\quad +O\Big(\int_{B^c(R)}\big(J_\alpha*|\cdot|^{-\tau}|u|^p\big)
 |x|^{-\tau}|u|^p\,dx\Big) \\
&= \frac{4(\alpha-N)}{p}\mathcal P[u]+O\Big(\int_{B^c(R)}
 \big(J_\alpha*|\cdot|^{-\tau}|u|^p\big)|x|^{-\tau}|u|^p\,dx\Big).
\end{aligned} 
\end{equation}
Collecting the above estimates, one obtains
\begin{equation} \label{36462}
(N)= -\frac {4B}{p}\mathcal P[u]+O\Big(\int_{B^c(R)}\big(J_\alpha*
 |\cdot|^{-\tau}|u|^p\big)|x|^{-\tau}|u|^p\,dx\Big).
\end{equation}
Hence, by \eqref{36461} and \eqref{36462}, one has
\begin{equation} \label{r2}
\begin{aligned}
M_R'
&\leq-\int_{\mathbb{R}^N}\Delta^2\xi_R|u|^2\,dx+8\int_{\mathbb{R}^N}|\nabla u|^2
 +8\lambda\int_{\mathbb{R}^N}\frac{|u|^2}{r^2}\,dx
 -\frac{4B}p\mathcal P[u] \\
&\quad +O\Big(\int_{B^c(R)}|x|^{-\tau}|u|^p\big(J_\alpha*
 |\cdot|^{-\tau}|u|^p\big)\,dx\Big) \\
&\leq 8\Big(\|\sqrt{\mathcal K_\lambda} u\|^2-\frac{B}{2p}\mathcal P[u]\Big)
 +O\Big(\int_{B^c(R)}|x|^{-\tau}|u|^p\big(J_\alpha*|\cdot|^{-\tau}|u|^p\big)
  \,dx\Big)+O(R^{-2}).
\end{aligned} 
\end{equation}
Now, with Lemma \ref{hls}, the H\"older and Gagliardo-Nirenberg 
inequalities via the mass conservation, one writes
\begin{equation} \label{1231-1}
\begin{aligned}
\int_{B^c(R)}|x|^{-\tau}|u|^p\big(J_\alpha*|\cdot|^{-\tau}|u|^p\big)\,dx 
&\lesssim \||x|^{-\tau}|u|^p\|_{\frac{2N}{\alpha+N}}\||x|^{-\tau}
 |u|^p\|_{L^\frac{2N}{\alpha+N}(B^c(R))} \\
&\lesssim R^{-\tau}\|u\|_{\frac{2Np}{\alpha+N}}^p\||x|^{-\tau}
 |u|^p\|_{\frac{2N}{\alpha+N}} \\
&\lesssim R^{-\tau}\|u_0\|^{p-\frac{N(p-1)-\alpha}2}
 \|\nabla u\|^{\frac{N(p-1)-\alpha}2}\||x|^{-\tau}
  |u|^p\|_{\frac{2N}{\alpha+N}} \\
&\lesssim R^{-\tau}\|\nabla u\|^{\frac{N(p-1)-\alpha}2}\||x|^{-\tau}
|u|^p\|_{\frac{2N}{\alpha+N}}.
\end{aligned}
\end{equation}
Now, using Proposition \ref{gag2}, with $\frac{2N\tau}{\alpha+N}$ 
instead of $2\tau$ and $\frac{2Np}{\alpha+N}$ instead of $2q$, via 
the fact that $0<\tau<1+\frac\alpha N$ and $p<p^c$, one has
\begin{equation} \label{r1}
\begin{aligned}
\|\,|x|^{-\tau}|u|^p\|_{\frac{2N}{\alpha+N}}
&=\Big(\int_{\mathbb{R}^N} |x|^{-\frac{2N\tau}{\alpha+N}}
 |u|^{\frac{2Np}{\alpha+N}}\,dx\Big)^\frac{\alpha+N}{2N} \\
&\lesssim \|u_0\|^{p-(\frac{N(p-1)-\alpha}2+\tau)}
 \|\sqrt{\mathcal K_\lambda} u\|^{\frac{N(p-1)-\alpha}2+\tau}.
\end{aligned}
\end{equation}
Thus, recall that $B=Np-N-\alpha+2\tau$, by \eqref{r2} via \eqref{r1} 
and \eqref{norm}, one obtains for large $R{\gg}1$,
\begin{equation}
M_R'
\leq C\mathcal I[u]+\frac C{R^{{\tau}}}\|\sqrt{\mathcal 
 K_\lambda} u\|^{B-{\tau}}+\frac C{R^2}.\label{4.241}
\end{equation}
Since $\mathcal I[u]<0$, by Gagliardo-Nirenberg estimate in Proposition 
\ref{gag}, via the mass conservation law, one has
\begin{equation} \label{r3}
\|\sqrt{\mathcal K}_\lambda u\|^2
\lesssim\mathcal P[u] 
\lesssim\|\sqrt{\mathcal K_\lambda} u\|^{B}\|u\|^{2p-B} 
\lesssim\|\sqrt{\mathcal K_\lambda} u\|^{B}.
\end{equation}
Thus, by $B>2$, there exists $C_0>0$ such that for any $t\in[0, T^*)$,
\begin{equation} \label{371}
\|\sqrt{\mathcal K_\lambda} u(t)\|\geq C_0.
\end{equation}
Hence, \eqref{nw} and \eqref{4.241}-\eqref{4.251} give for $2<B\leq2+\tau$ 
and $R \gg 1$,
\begin{equation} \label{4.261}
\begin{aligned}
M'_R[u]
&\lesssim \mathcal I[u]+R^{-2}+R^{-\tau}\|\sqrt{\mathcal K_\lambda} 
 u\|^{B-\tau} \\
&\lesssim -\|\sqrt{\mathcal K_\lambda} u\|^2+R^{-2}+R^{-\tau}
 \|\sqrt{\mathcal K_\lambda} u\|^{B-\tau} \\
&\lesssim \|\sqrt{\mathcal K_\lambda} u\|^2\Big(-1+R^{-2}
 +R^{-\tau}\|\sqrt{\mathcal K_\lambda} u\|^{B-2-\tau}\Big) \\
&\lesssim -\|\sqrt{\mathcal K_\lambda} u\|^2.
\end{aligned}
\end{equation}
Time integration, \eqref{371}, and \eqref{4.261} imply that 
\begin{equation}\label{4.281}
M_R[u(t)]\lesssim-t,\quad t>T>0.   
\end{equation}
By time integration again, from \eqref{4.261}, it follows that
\begin{equation} 
M_R[u(t)]
\lesssim -\int_T^t\|\sqrt{\mathcal K_\lambda} u(s)\|^2\,ds\label{4.271}.
\end{equation}
Now, the definition \eqref{mrwtz} via \eqref{norm} gives
\begin{equation}
|M_R|
= 2|\Im\int_{\mathbb{R}^N} \bar u(\nabla\xi_R\cdot\nabla u)\,dx| 
\lesssim R\|\nabla u\|\|u\| 
\lesssim R\|\sqrt{\mathcal K_\lambda} u\| \label{4.28}.
\end{equation}
Thus, by \eqref{4.281}, \eqref{4.271} and \eqref{4.28}, it follows that
\begin{equation}\label{4.29}
\int_T^t\|\sqrt{\mathcal K_\lambda} u(s)\|^2\,ds
\lesssim |M_R[u(t)]|
\lesssim R\|\sqrt{\mathcal K_\lambda} u(t)\|,\quad \forall t>T.
\end{equation}
Take the real function $f(t):=\int_T^t\|\sqrt{\mathcal K_\lambda} u(s)\|^2$. 
By \eqref{4.29}, one obtains $f^2\lesssim f'$. This ODI has no global solution. 
Indeed, for $T'>T>t$, an integration gives
$$
t-T'\lesssim\int_{T'}^t\frac{f'(s)}{f^2(s)}\,ds
=\frac1{f(T')}-\frac1{f(t)}\leq\frac1{f(T')}.
$$
This implies $T'+\frac c{f(T')}$. This completes the proof.

\subsubsection*{Second case}
Assume that \eqref{t13} holds. Taking account of the previous sub-section, 
it is sufficient to prove the next result. 

\begin{lemma}\label{ky1}
There exist $C>0$ and $\varepsilon>0$, such that for any $t\in [0, T^*)$, 
the following statements hold:
$$
 \mathcal{I}[u(t)]<-C<0,
$$
and
\begin{align}
 \mathcal I[u(t)]+\varepsilon \|\sqrt{\mathcal K_\lambda}u(t)\|^2<0.\label{nw1}
\end{align}
\end{lemma}

\begin{proof}
(1)
Define the quantity
$\mathcal C:=\frac{C_{N,p,\tau,\lambda}}{p}\| u\|^{A}$.
Then, by Proposition \ref{gag}, one writes
\begin{equation}\label{xxxx1}
F(\|\sqrt{\mathcal K_\lambda} u(t)\|^2)
:=\|\sqrt{\mathcal K_\lambda} u(t)\|^2 
- \mathcal C\|\sqrt{\mathcal K_\lambda} u(t)\|^{{B}}\leq \mathcal E[u_0],
 \quad \text{on}\quad[0,T^*).
\end{equation}
Now, since $B>2,$ the above real function has a maximum 
$$
F(x_1):=F\Big[\Big( \frac{2}{\mathcal CB}\Big)^{\frac{2}{B-2}}\Big]
=\Big( \frac{2}{\mathcal CB}\Big)^{\frac{2}{B-2}}\Big( 1-\frac{2}{B}\Big) .
$$
Moreover, thanks to Pohozaev identities \eqref{poh} and the condition 
\eqref{t13}, it follows that
\begin{equation}\label{x1}
\mathcal E[u_0]< \frac{B-2}{A}\big(\mathcal M[u_0]\big)^{-\alpha_c}
\big(\mathcal M[{\varphi}]\big)^{\frac{1}{s_c}}.
\end{equation}
In addition, by \eqref{part3}, one obtains
\begin{equation} \label{xx1}
\begin{aligned}
F(x_1)
&= \Big(\frac2{B\mathcal C}\Big)^\frac2{B-2}\Big(1-\frac2{B}\Big) \\
&= \Big((\frac{A}{B})^{1-\frac{B}2}(\mathcal M[{\varphi}])^{p-1}
 (\mathcal M[u_0])^{-\frac{A}2}\Big)^\frac2{B-2}\Big(1-\frac2{B}\Big) \\
&= \frac{B-2}{A}\Big((\mathcal M[u_0])^{-\frac{A}2}
 (\mathcal M[{\varphi}])^{p-1}\Big)^\frac2{B-2} \\
&=  \frac{B-2}{A}\big(\mathcal M[u_0]\big)^{-\alpha_c}
 \big( \mathcal M[{\varphi}]\big)^{\frac{1}{s_c}}.
\end{aligned}
\end{equation}
Relations \eqref{x1} and \eqref{xx1} imply that
$\mathcal E[u_0]<F(x_1)$.
By the previous inequality and \eqref{xxxx1}, one has
\begin{equation}\label{ss31}
 F\Big(\|\sqrt{\mathcal K_\lambda} u(t)\|^2\Big)
 \leq \mathcal E[u_0]<F(x_1). 
\end{equation}
Direct calculations show that
$$
x_1=\frac{B}{A}\big(\mathcal M[{\varphi}]
\big)^{\frac{1}{s_c}}\big(\mathcal M[u_0]\big)^{-\alpha_c},
$$
Now, via \eqref{poh}, the inequality \eqref{t13} reads 
$$
\|\sqrt{\mathcal K_\lambda} u_0\|^2
>\frac{B}{A}\mathcal M[{\varphi}]
\Big(\frac{\mathcal M[{\varphi}]}{\mathcal M[u_0]}\Big)^{\alpha_c}=x_1.
$$
Thus, the continuity in time with \eqref{ss31} gives
$$ 
\|\sqrt{\mathcal K_\lambda} u (t)\|^2>x_1,\quad \forall\, t\in [0,T^*) .
$$
Then, by \eqref{poh}, it follows that
\begin{equation}
\mathcal{MG}[u(t)]>1, \quad\text{on } [0,T^*) .\label{stbb}
\end{equation}
Thus, by the Pohozaev identity 
$B\mathcal E[{\varphi}]=(B-2)\|\sqrt{\mathcal K_\lambda}{\varphi}\|^2$, 
it follows that
\begin{align*}
\mathcal I[u][\mathcal M[u]]^{\alpha_c}
&= \Big(\|\sqrt{\mathcal K_\lambda} u\|^2-\frac{B}{2q}\mathcal P[u]\Big)[\mathcal M[u]]^{\alpha_c}\\
&= \frac{B}{2}\mathcal E[u][\mathcal M[u]]^{\alpha_c}-(\frac{B}{2}-1)\|\sqrt{\mathcal K_\lambda} u\|^2[\mathcal M[u]]^{\alpha_c}\\
&\leq \frac{B}{2}(1-\nu)\mathcal E[{\varphi}][\mathcal M[{\varphi}]]^{\alpha_c}-(\frac{B}{2}-1)\|\sqrt{\mathcal K_\lambda}{\varphi}\|^2[\mathcal M[{\varphi}]]^{\alpha_c}\\
&\leq -\nu(\frac{B}{2}-1)\|\sqrt{\mathcal K_\lambda}{\varphi}\|^2[\mathcal M[\varphi]]^{\alpha_c}.
\end{align*}
The proof of the first point is complete.

(2) Assume that \eqref{nw1} fails, then there exists a time sequence
 $\{t_n\}\subset [0, T^*)$ such that
\begin{equation} \label{rv6}
   -\varepsilon_n\big(\frac B2-1\big)\|\sqrt{\mathcal{K}_\lambda}u(t_n)\|^2
< \mathcal{I}[u(t_n)]<0,
\end{equation}
where $\varepsilon_n\to0$ and $n\to \infty$. Moreover, note that
$$
2 \mathcal{I}[u(t_n)]=B \mathcal E[u(t_n)]-(B-2)
\|\sqrt{\mathcal{K}_\lambda}u(t_n)\|^2.
$$
Hence, \eqref{rv6} implies that
\begin{equation}
 (1-\varepsilon_n)\big(1-\frac2B\big)\|\sqrt{\mathcal{K}_\lambda}u(t_n)\|^2
 <\mathcal E[u_0].\label{rv7}
\end{equation}
Hence, by \eqref{poh}, \eqref{t13}, \eqref{stbb} and \eqref{rv7},
 we obtain
\begin{equation} \label{rv8}
\begin{aligned}
\mathcal E[u_0]\mathcal M[u_0]^{\alpha_c}
 &>(1-\varepsilon_n)\big(1-\frac2B\big)
  \|\sqrt{\mathcal{K}_\lambda}u(t_n)\|^2\mathcal M[u_0]^{\alpha_c} \\
 &>(1-\varepsilon_n)\big(1-\frac2B\big)
  \|\sqrt{\mathcal{K}_\lambda}\varphi\|^2\mathcal M[\varphi]^{\alpha_c} \\
 &>(1-\varepsilon_n)\mathcal E[\varphi]
  \mathcal M[\varphi]^{\alpha_c}.
\end{aligned}
\end{equation}
Taking $n\to\infty$ in \eqref{rv8}, yields 
\begin{equation} \label{rv9}
\mathcal E[u_0]\mathcal M[u_0]^{\alpha_c}
 \geq\mathcal E[\varphi]\mathcal M[\varphi]^{\alpha_c}. 
\end{equation}
The proof of the second statement \eqref{nw1} is achieved by the 
contradiction of \eqref{rv9} with $\mathcal{ME}[u_0]<1$ in \eqref{t13}. 
Hence, this lemma is established.
\end{proof}

\subsection{Bi-harmonic case}
In this sub-section, one assumes that $s=2$. 

\subsubsection*{First case}
Assume that \eqref{ss} holds. We start with the next auxiliary result.

\begin{lemma}\label{2ky}
\begin{enumerate}
\item The set $\mathcal{A}^-$ is stable under the flow of \eqref{S}.
\item There exists $\varepsilon>0$, such that for any $t\in [0, T^*)$
 \begin{align}
 \mathcal I[u(t)]+\varepsilon \|{\Delta}u(t)\|^2&\leq-\frac B4\big(m-\mathcal{S}[u(t)]\big).\label{2nw}
 \end{align}
 \end{enumerate}
\end{lemma}

\begin{proof}
(1) The proof follows a similar approach to the first point in Lemma 
\ref{ky}.

(2) Now, taking the scaling $u_\rho:=\rho^{\frac N2}u(\rho\cdot)$ for 
$\rho>0$, we compute 
\begin{gather}
\|u_\rho\|=\|u\|;\label{2scl1}\\
\|{\Delta}u_\rho\|= \rho^2\|{\Delta}u\|;\label{2scl2}\\
\mathcal P[u_\rho]= \rho^{2B}\mathcal P[u]\label{2scl4}.
\end{gather}
Moreover, taking the real function 
${\Upsilon}:\rho\mapsto \mathcal{S}[u_\rho]$, 
we obtain ${\Upsilon}(\rho)=\rho^4\|{\Delta}u\|^2+\|u\|^2
-\frac{\rho^{2B}}{p}\mathcal P[u]$ and the first derivative reads
\begin{equation}
 {\Upsilon}'(\rho)=4\rho^3\|{\Delta}u\|^2-2B\frac{\rho^{2B-1}}{p}
 \mathcal P[u]=4\rho^{-1}\mathcal I[u_\rho].\label{2rv1}
\end{equation} 
This implies
\begin{equation}
 \rho{\Upsilon}'(\rho)=4\rho^4\|{\Delta}u\|^2-2B\frac{\rho^{2B}}{p}\mathcal P[u]=4\mathcal I[u_\rho].\label{2rv2}
\end{equation}
Moreover, since $B>2$, we obtain
\begin{equation} \label{2rv3}
\begin{aligned}
 \big(\rho{\Upsilon}'(\rho)\big)'
&=16\rho^3\|{\Delta}u\|^2-4B^2\frac{\rho^{2B-1}}{p}\mathcal P[u] \\
&=2B{\Upsilon}'(\rho)-8(B-2)\rho^3\|{\Delta}u\|^2 \\
&\leq 2B{\Upsilon}'(\rho).
\end{aligned}
\end{equation}
Now, we claim that there exists $\rho_0\in(0,1)$ such that
\begin{equation}
   \mathcal I[u_{\rho_0}]=0.\label{2claim}
\end{equation}
Indeed, by \eqref{2scl2} and \eqref{2scl4}, we have
\begin{equation}
 \mathcal I[u_{\rho}]
 =\rho^4\Big(\|{\Delta}u\|^2-\frac{\rho^{2(B-2)}}{2p}\mathcal P[u]\Big) 
 :=\rho^4\Xi(\rho).\label{2rv4}
\end{equation}
Note that $\Xi(0)>0$ and $\Xi(1)=\mathcal I[u]<0$. Then there exists 
$\rho_0\in(0,1)$ such that $\Xi(\rho_0)=0$, the claim is proved. 
Hence, by \eqref{2rv1}, we have 
${\Upsilon}'(\rho_0)=0$ and ${\Upsilon}(\rho_0)=\mathcal S[u_{\rho_0}]\geq m$. 
Hence, an integration of \eqref{2rv3} on $[\rho_0,1]$ gives
$$
\Upsilon'(1)-\rho_0{\Upsilon}'(\rho_0)
\leq 2B {\Upsilon}(1)-2B {\Upsilon}(\rho_0).
$$
Note that $\Upsilon'(1)=4\mathcal I[u]$, 
$\rho_0{\Upsilon}'(\rho_0)=4\mathcal I[u(\rho_0)]=0$, and 
${\Upsilon}(1)=\mathcal{S}[u]$, the above inequality further implies 
\begin{equation}
 \mathcal{I}[u] \leq-\frac B2\big(m-\mathcal{S}[u]\big).\label{2rv5}
\end{equation}
On the other hand, we write
\begin{equation}
 \|{\Delta}u\|^2
 =\frac{B}{B-2}\big(\mathcal{S}[u]-\frac2B\mathcal{I}[u]-\|u\|^2\big).
\end{equation}
Hence, by \eqref{2rv5}, we have that there exists $0<\varepsilon\ll1$, 
such that
\begin{equation} \label{24.251}
\begin{aligned}
\mathcal{I}[u]+\varepsilon\|{\Delta}u\|^2
 &=\Big(1-\frac{2\varepsilon}{B-2}\Big)\mathcal{I}[u]
  +\varepsilon\frac{B}{B-2} \big(\mathcal{S}[u]-\|u\|^2\big) \\
 &\leq-\frac B2\Big(1-\frac{2\varepsilon}{B-2}\Big)
  \big(m-\mathcal{S}[u]\big)+\varepsilon\frac{B}{B-2} \mathcal{S}[u] \\
 &\leq-\frac B4\big(m-\mathcal{S}[u]\big).
\end{aligned}
\end{equation}
The last statement of Lemma \ref{2ky} is proved by \eqref{24.251}.
\end{proof}

Now we turn to the proof of the main results. 
Using the estimate $\|\nabla^\gamma \xi_R\|_\infty\lesssim R^{2-|\gamma|}$, 
one has
\begin{gather}
|\int_{\mathbb{R}^N}\Delta^2\xi_R|\nabla u|^2\,dx|
 +|\int_{\mathbb{R}^N}\partial_{jk}\Delta \xi_R\partial_ju\partial_k\bar u\,dx|
 \lesssim R^{-2}\|\nabla u\|^2;\label{mr1}\\
\big|\int_{\mathbb{R}^N}(\Delta^3\xi_R)|u|^2\,dx\big|
\lesssim R^{-4}.\label{mr2}
\end{gather}
Using estimates \eqref{mr1} and \eqref{mr2} via Morawetz identity 
\eqref{mrw1}, one obtains
\begin{equation} \label{mr3-1}
\begin{aligned}
M_R'
&= \frac4{p}\int_{\mathbb{R}^N}\partial_k\xi_R\partial_k
 \Big[\big(J_\alpha*|\cdot|^{-\tau}|u|^p\big)|x|^{-\tau}\Big]|u|^p\,dx
  +O(R^{-4})+\|\nabla u\|^2O(R^{-2}) \\
&\quad -4{N}(1-\frac2p)\int_{B(R)}\big(J_\alpha*|
 \cdot|^{-\tau}|u|^p\big)|x|^{-\tau}|u|^p\,dx \\
&\quad -2\Big((1-\frac2p)\int_{B^c(R)}\Delta \xi_R
 \big(J_\alpha*|\cdot|^{-\tau}|u|^p\big)|x|^{-\tau}|u|^p\,dx
 -4\int_{\mathbb{R}^N}\partial_{jk}\xi_R\partial_{ik}u\partial_{ij}\bar u\,dx\Big). 
\end{aligned}
\end{equation}
Denoting the partial derivative $\frac\partial{\partial_{xi}}u:=u_i$, 
one obtains via \eqref{symm},
\begin{equation} \label{mr3}
\begin{aligned}
\int_{\mathbb{R}^N}\partial_{jk}\xi_R\partial_{ik}u\partial_{ij}\bar u\,dx 
&= \int_{\mathbb{R}^N}\Big[\Big(\frac{\delta_{jk}}{|x|}-\frac{x_jx_k}{{|x|}^3}\Big)
 \partial_r\xi_R+\frac{x_jx_k}{{|x|}^2}\partial_r^2\xi_R\Big]
 \partial_{ik}u\partial_{ij}\bar u\,dx \\
&= \sum_{i=1}^N\int_{\mathbb{R}^N}|\nabla u_i|^2\frac{\partial_r\xi_R}{|x|}\,dx
 +\sum_{i=1}^N\int_{\mathbb{R}^N}\frac{|x\cdot\nabla u_i|^2}{|x|^2}
 \Big(\partial_r^2\xi_R-\frac{\partial_r\xi_R}{|x|}\Big)\,dx.
\end{aligned}
\end{equation}
From \eqref{mr3-1} and \eqref{mr3}, via the equality
 $\sum_{i=1}^N\|\nabla u_i\|^2=\|\Delta u\|^2$, it follows that
\begin{align*}
M_R'
&= \frac4{p}\int_{\mathbb{R}^N}\partial_k\xi_R\partial_k
 \Big[\big(J_\alpha*|\cdot|^{-\tau}|u|^p\big)|x|^{-\tau}\Big]|u|^p\,dx
  +O(R^{-4})+\|\nabla u\|^2O(R^{-2})\\
&\quad +16\|\Delta u\|^2-4{N}(1-\frac2p)\int_{B(R)}\big(J_\alpha*|
 \cdot|^{-\tau}|u|^p\big)|x|^{-\tau}|u|^p\,dx\\
&\quad -2(1-\frac2p)\int_{B^c(R)}\Delta \xi_R
 \big(J_\alpha*|\cdot|^{-\tau}|u|^p\big)|x|^{-\tau}|u|^p\,dx\\
&\quad +8\Big(\sum_{i=1}^N\int_{\mathbb{R}^N}|\nabla u_i|^2
 \big(\frac{\partial_r\xi_R}{|x|}-2\big)\,dx+\sum_{i=1}^N
 \int_{\mathbb{R}^N}\frac{|x\cdot\nabla u_i|^2}{|x|^2}\big(\partial_r^2\xi_R
 -\frac{\partial_r\xi_R}{|x|}\big)\,dx\Big).
\end{align*}
Then, \eqref{prpr} gives
\begin{equation} \label{5.48}
\begin{aligned}
M_R'
&\leq \frac4{p}\int_{\mathbb{R}^N}\partial_k\xi_R\partial_k
 \Big[\big(J_\alpha*|\cdot|^{-\tau}|u|^p\big)|x|^{-\tau}
 \Big]|u|^p\,dx+cR^{-2}(R^{-2}+\|\nabla u\|^2) \\
&\quad +16\|\Delta u\|^2-4{N}(1-\frac2p)\int_{B(R)}
 \big(J_\alpha*|\cdot|^{-\tau}|u|^p\big)|x|^{-\tau}|u|^p\,dx \\
&\quad -2(1-\frac2p)\int_{B^c(R)}\Delta \xi_R\big(J_\alpha*|
 \cdot|^{-\tau}|u|^p\big)|x|^{-\tau}|u|^p\,dx.
\end{aligned}
\end{equation}
Take the quantity 
\begin{align*}
(A)
&:= \int_{\mathbb{R}^N}\partial_k\xi_R\partial_k
\Big[\big(J_\alpha*|\cdot|^{-\tau}|u|^p\big)|x|^{-\tau}\Big]|u|^p\,dx\\
&= (\alpha-N)\int_{\mathbb{R}^N}\nabla \xi_R\Big(\frac\cdot{|x|^2}J_\alpha*|
 \cdot|^{-\tau}|u|^p\Big)|x|^{-\tau}|u|^p\,dx\\
&\quad -\tau\int_{\mathbb{R}^N}\frac{\nabla \xi_R\cdot x}{|x|^2}
 \big(J_\alpha*|\cdot|^{-\tau}|u|^p\big)|x|^{-\tau}|u|^p\,dx\\
&:=(\alpha-N)\cdot(I)-\tau\cdot(II).
\end{align*}
In the same way as \eqref{372}, one has
\begin{align*}
(I)
&= \int_{B(R)\times B(R)}J_\alpha(x-y)|y|^{-\tau}|u(y)|^p
 |x|^{-\tau}|u(x)|^p\,dx\,dy\\
&\quad +O\Big(\int_{B^c(R)}\big(J_\alpha*|\cdot|^{-\tau}|u|^p\big)
 |x|^{-\tau}|u|^p\,dx\Big)\\
&= \int_{B(R)}\big(J_\alpha*|\cdot|^{-\tau}|u|^p\big)|x|^{-\tau}
 |u(x)|^p\,dx+O\Big(\int_{B^c(R)}\big(J_\alpha*|\cdot|^{-\tau}
 |u|^p\big)|x|^{-\tau}|u|^p\,dx\Big).
\end{align*}
From the properties of $\xi_R$, one writes
\[
(II)
= 2\int_{B(R)}\big(J_\alpha*|\cdot|^{-\tau}|u|^p\big)
|x|^{-\tau}|u|^p\,dx+O\Big(\int_{B^c(R)}\big(J_\alpha*
|\cdot|^{-\tau}|u|^p\big)|x|^{-\tau}|u|^p\,dx\Big).
\]
Thus, 
\begin{align*}
(A)
&= 2(-\tau -\frac{N-\alpha}2)\int_{B(R)}\big(J_\alpha*
 |\cdot|^{-\tau}|u|^p\big)|x|^{-\tau}|u(x)|^p\,dx\\
&\quad +O\Big(\int_{B^c(R)}\big(J_\alpha*|\cdot|^{-\tau}|u|^p\big)
 |x|^{-\tau}|u|^p\,dx\Big).
\end{align*}
Further, \eqref{5.48} implies that
\begin{equation} \label{3.40}
\begin{aligned}
M_R'
&\leq 2\Big(8\int_{\mathbb{R}^N}|\Delta u|^2\,dx-2{N}(1-\frac2p)
 \int_{\mathbb{R}^N}\big(J_\alpha*|\cdot|^{-\tau}|u|^p\big)|x|^{-\tau}|u|^p\,dx\Big) \\
&\quad +\frac4{p}(A)+cR^{-2}(R^{-2}+\|\nabla u\|^2)+O\Big(\int_{B^c(R)}
 \big(J_\alpha*|\cdot|^{-\tau}|u|^p\big)|x|^{-\tau}|u|^p\,dx\Big) \\
&= 2\Big(8\int_{\mathbb{R}^N}|\Delta u|^2\,dx-2{N}(1-\frac2p)\int_{\mathbb{R}^N}
 \big(J_\alpha*|\cdot|^{-\tau}|u|^p\big)|x|^{-\tau}|u|^p\,dx\Big) \\
&\quad +\frac8{p}(-\tau -\frac{N-\alpha}2)\int_{\mathbb{R}^N}
 \big(J_\alpha*|\cdot|^{-\tau}|u|^p\big)|x|^{-\tau}|u(x)|^p\,dx \\
&\quad +O(R^{-2})+O\Big(\int_{B^c(R)}\big(J_\alpha*|\cdot|^{-\tau}
 |u|^p\big)|x|^{-\tau}|u|^p\,dx\Big) \\
&= 16\mathcal I[u]+cR^{-2}(R^{-2}+\|\nabla u\|^2)+O\Big(\int_{B^c(R)}
 \big(J_\alpha*|\cdot|^{-\tau}|u|^p\big)|x|^{-\tau}|u|^p\,dx\Big).
\end{aligned}
\end{equation}
Since $0<\tau<s\big(1+\frac\alpha N\big)$, using the Gagliardo-Nirenberg 
estimate in Proposition \ref{gag2} via the mass conservation, one writes 
\begin{equation} \label{r5}
\begin{aligned}
\||x|^{-\tau}u^p\|_{L^{\frac{2N}{\alpha+N}}}
&=\Big(\int_{\mathbb{R}^N}|x|^{-\frac{2N\tau}{\alpha+N}}|u|^{\frac{2Np}{\alpha+N}}
 \,dx\Big)^{\frac{\alpha+N}{2N}} \\
&\lesssim\|u\|^{p-(\frac{Np-N-\alpha+2\tau}4)}
 \|\Delta u\|^{\frac{Np-N-\alpha+2\tau}4} \\
&\lesssim\|\Delta u\|^{\frac{Np-N-\alpha+2\tau}4}.
\end{aligned}
\end{equation}
Now, by the same way as in \eqref{1231-1}, one obtains
\begin{equation} \label{3.23}
\begin{aligned}
\int_{B^c(R)}\big(J_\alpha*|\cdot|^{-\tau}|u|^p\big)|x|^{-\tau}|u|^p\,dx
&\lesssim\||x|^{-\tau}u^p\|_{{\frac{2N}{\alpha+N}}}\|
 |x|^{-\tau}u^p\|_{L^{\frac{2N}{\alpha+N}}(B^c(R))} \\
&\lesssim R^{-\tau}\|u\|_{\frac{2Np}{\alpha+N}}^p\|
 |x|^{-\tau}u^p\|_{L^{\frac{2N}{\alpha+N}}} \\
&\lesssim R^{-\tau}\|\Delta u\|^{\frac{Np-N-\alpha}4}
 \|\Delta u\|^{\frac{Np-N-\alpha+2\tau}4} \\
&\lesssim R^{-\tau}\|\Delta u\|^{\frac{Np-N-\alpha+\tau}2},
\end{aligned}
\end{equation}
where $\frac{Np-N-\alpha+\tau}2 =B-\frac\tau2\in (0, 2]$. 
Thus, by the interpolation  $\|\nabla u\|^2\lesssim\|\Delta u\|\|u\|$ 
and Young's estimate via \eqref{3.40} and \eqref{3.23} one obtains 
\begin{equation}\label{3.24}
M_R'
\lesssim\mathcal I[u]+R^{-\tau}\|\Delta u\|^{B-\frac\tau2}
+R^{-2}\|\Delta u\|^2+R^{-2}.
\end{equation}
Since $\mathcal I[u]<0$, by Gagliardo-Nirenberg estimate in Proposition 
\ref{gag} via the mass conservation law, one has
\[
\|\Delta u\|^2
\lesssim \mathcal P[u]
\lesssim \|\Delta u\|^{B}\|u\|^{2p-B}
\lesssim \|\Delta u\|^{B}.
\]
Thus, $B>2$ implies that there exists $C_1>0$ such that for any 
$t\in[0, T^*)$,
\begin{equation} \label{3522}
\|\Delta u(t)\|\geq C_1.
\end{equation}
Further, \eqref{2nw}, \eqref{3.24} and \eqref{3522},  for
 $2<B\leq2+\frac\tau2$ and $R\gg 1$, give
\begin{equation} \label{3.26}
\begin{aligned}
 M_R'
&\lesssim \mathcal I[u]+R^{-2}+R^{-2}\|\Delta u\|^2+R^{-\tau}
 \|\Delta u\|^{B-\frac\tau2} \\
&\lesssim -\|\Delta u\|^2+R^{-2}+R^{-2}\|\Delta u\|^2+R^{-\tau}
 \|\Delta u\|^{B-\frac\tau2} \\
&\lesssim \|\Delta u\|^2\Big(-1+R^{-2}+R^{-\tau}
 \|\Delta u\|^{B-2-\frac\tau2}\Big) \\
&\lesssim -\|\Delta u\|^2.
\end{aligned}
\end{equation}
By time integration, \eqref{3522} and \eqref{3.26} imply that 
\begin{equation} \label{3.25'}
M_R[u(t)]\lesssim-t, \quad t>T>0. 
\end{equation}
By time integration again, from \eqref{3.26}, it follows that
\begin{equation}
M_R[u(t)]
\lesssim -\int_T^t\|\Delta u(s)\|^2\,ds,\quad\forall t>T \label{3.27}.
\end{equation}
Now, the definition \eqref{mrwtz} and an interpolation argument give
\begin{equation}
|M_R[u]|
= 2|\Im\int_{\mathbb{R}^N} \bar u(\nabla\xi_R\cdot\nabla u)\,dx| 
\lesssim R\|\nabla u\|\|u\| 
\lesssim R\|\Delta u\|^{1/2} \label{3.28}.
\end{equation}
So, by \eqref{3.25'}, \eqref{3.27} and \eqref{3.28}, it follows that
\begin{equation} \label{3.29}
\int_T^t\|\Delta u(s)\|^2\,ds\lesssim| M_R[u(t)]|
\lesssim R\|\Delta u(t)\|^{1/2},\quad \forall t>T.
\end{equation} 
Take the real function $f(t):=\int_T^t\|\Delta u(s)\|^2$. 
By \eqref{3.29}, one obtains $f^4\lesssim {f'}$. 
Like previously, this ODI has no global solution. This completes the proof.

\subsubsection*{Second case}

Assume that \eqref{t13} holds. It is sufficient to prove the next 
intermediate result. 

\begin{lemma}\label{ky1'}
There exist $C>0$ and $\varepsilon>0$, such that for any $t\in [0, T^*)$, 
the following statements hold:
\begin{gather}
\mathcal{I}[u(t)]<-C<0, \nonumber\\
\mathcal I[u]+\varepsilon \|\Delta u\|^2<0.\label{nw11}
\end{gather}
\end{lemma}

\begin{proof}
(1)
Define the quantity
$\mathcal C:=\frac{C_{N,p,\tau,\alpha,\lambda}}{p}\| u\|^{A}$.
Then, by Proposition \ref{gag}, one writes
\begin{equation}\label{xxxx}
F(\|\Delta u(t)\|^2):=\|\Delta u(t)\|^2 
- \mathcal C\|\Delta u(t)\|^{B}\leq \mathcal E[u_0], 
\quad \text{on }[0,T^*).
\end{equation}
Now, since $p>p_c$ gives $B>2,$ the above real function $F$ has a 
maximum 
$$
F(x_1):=F\Big[\Big( \frac{2}{\mathcal C B}\Big)^{\frac{2}{ B-2}}\Big]
=\Big( \frac{2}{\mathcal C B}\Big)^{\frac{2}{B-2}}\Big( 1-\frac{2}{B}\Big) .
$$
Moreover, thanks to Pohozaev identities \eqref{poh} and  condition 
\eqref{t13}, it follows that
\begin{equation}\label{x}
\mathcal E[u_0]
< \frac{B-2}{A }\big(\mathcal M[u_0]\big)^{-\alpha_c}
\big(\mathcal M[\varphi]\big)^{2/s_c}.
\end{equation}
In addition, by \eqref{part3} and the equality $s_c=\frac{B-2}{p-1}$, 
one obtains
\begin{equation} \label{xx}
\begin{aligned}
F(x_1)
&= \Big(\frac2{B \mathcal C}\Big)^\frac2{B-2}\Big(1-\frac2{B }\Big) \\
&= \Big((\frac{A }{B })^{1-\frac{B }2}(\mathcal M[\varphi])^{p-1}
 (\mathcal M[u_0])^{-A/2}\Big)^\frac2{B-2}\Big(1-\frac2{B }\Big) \\
&= \frac{B-2}{A }\Big((\mathcal M[u_0])^{-A/2}
 (\mathcal M[\varphi])^{p-1}\Big)^\frac2{B-2} \\
&=  \frac{B-2}{A }\big(\mathcal M[u_0]\big)^{-\alpha_c}
 \big( \mathcal M[\varphi]\big)^{2/s_c}.
\end{aligned}
\end{equation}
Relations \eqref{x} and \eqref{xx} imply that
$\mathcal E[u_0]<F(x_1)$.
By the previous inequality and \eqref{xxxx}, one has
\begin{equation}\label{ss3}
 F\Big(\|\Delta u(t)\|^2\Big) \leq\mathcal E[u_0]<F(x_1). 
\end{equation}
Direct calculations show that
$$
x_1=\frac{B }{A }\big(\mathcal M[\varphi]\big)^{2/s_c}
\big(\mathcal M[u_0]\big)^{-\alpha_c}.$$
Now, the inequality \eqref{t13} reads via \eqref{poh},
$$
\|\Delta u_0\|^2>\frac{B }{A }\mathcal M[\varphi]
\Big(\frac{\mathcal M[\varphi]}{\mathcal M[u_0]}\Big)^{\alpha_c}=x_1.
$$
Thus, the continuity in time with \eqref{ss3} give
$$ 
\|\Delta u (t)\|^2>x_1,\quad \forall\, t\in [0,T^*) .
$$
Further, one has 
\begin{equation}
 \mathcal{MG}[u(t)]>1,\quad\text{for all } t\in[0,T^*).\label{stb}
\end{equation}
Now, by Pohozaev identity \eqref{poh}, one has 
$B\mathcal E[\varphi]=(B-2)\|\Delta\varphi\|^2$. So, it follows that 
for some $0<\nu<1$,
\begin{align*}
\mathcal I[u][\mathcal M[u]]^{\alpha_c}
&= \Big(\|\Delta u\|^2-\frac{B }{2p}\mathcal P[u]\Big)
 [\mathcal M[u]]^{\alpha_c}\\
&= \frac{B}{2}\mathcal E[u][\mathcal M[u]]^{\alpha_c}
 -(\frac{B }{2}-1)\|\Delta u\|^2[\mathcal M[u]]^{\alpha_c}\\
&\leq \frac{B }{2}(1-\nu)\mathcal E[\varphi][\mathcal M[\varphi]
 ]^{\alpha_c}-(\frac{B }{2}-1)\|\Delta\varphi\|^2
  [\mathcal M[\varphi]]^{\alpha_c}\\
&\leq -\nu(\frac{B }{2}-1)\|\Delta\varphi\|^2[\mathcal M
 [\varphi]]^{\alpha_c}.
\end{align*}
The proof of the first point is complete.

(2) Assume that \eqref{nw11} fails, then there exists a time sequence 
$\{t_n\}\subset [0, T^*)$, such that
\begin{equation}
   -\varepsilon_n\big(\frac B2-1\big)\|{\Delta}u(t_n)\|^2
   < \mathcal{I}[u(t_n)]<0,\label{2rv6}
\end{equation}
where $\varepsilon_n\to0$ as $n\to \infty$. Moreover, note that
$$
2 \mathcal{I}[u(t_n)]=B \mathcal E[u(t_n)]-(B-2)\|{\Delta}u(t_n)\|^2.
$$
Hence, \eqref{2rv6} implies that
\begin{equation}
 (1-\varepsilon_n)\big(1-\frac2B\big)\|{\Delta}u(t_n)\|^2
 <\mathcal E[u_0]. \label{2rv7}
\end{equation}
Further,  by \eqref{poh}, \eqref{t13}, \eqref{stb} and \eqref{2rv7}, 
we obtain
\begin{equation}  \label{2rv8}
\begin{aligned}
\mathcal E[u_0]\mathcal M[u_0]^{\alpha_c}
 &>(1-\varepsilon_n)\big(1-\frac2B\big)\|{\Delta}u(t_n)\|^2
 \mathcal M[u_0]^{\alpha_c} \\
 &>(1-\varepsilon_n)\big(1-\frac2B\big)\|{\Delta}\varphi\|^2
  \mathcal M[\varphi]^{\alpha_c} \\
 &>(1-\varepsilon_n)\mathcal E[\varphi]\mathcal M[\varphi]^{\alpha_c}.
\end{aligned}
\end{equation}
Taking $n\to\infty$ in \eqref{2rv8}, we obtain
\begin{equation}
\mathcal E[u_0]\mathcal M[u_0]^{\alpha_c}
 \geq\mathcal E[\varphi]\mathcal M[\varphi]^{\alpha_c} \label{2rv9}.
\end{equation}
The proof of \eqref{2nw} is achieved by the contradiction of 
\eqref{2rv9} with \eqref{t13}. Hence, this lemma is established.
\end{proof}

\section{Schr\"odinger equation with local source term}

In this section, we establish Theorem \ref{t1'}.

\subsection{Schr\"odinger equation with inverse square potential}
In this subsection, we take $s=1$. 
\subsubsection*{First case}
One keeps previous notation and assume that \eqref{ss'} holds. 
We start with the next auxiliary result which can be proved arguing as 
in Lemma \ref{ky}.

\begin{lemma}\label{ky'}
 \begin{enumerate}
 \item The set $\mathcal{A'}^-$ is stable under the flow of \eqref{S'}.
 \item There exists $\varepsilon>0$, such that for any $t\in [0, T^*)$,
\begin{equation}
 \mathcal J[u(t)]+\varepsilon \|\sqrt{\mathcal K_\lambda}u(t)\|^2
 \leq-\frac{B'}4\big(m'-\mathcal{S'}[u(t)]\big).\label{nw'}
\end{equation}
\end{enumerate}
\end{lemma}

Proposition \ref{mrwz1} via \eqref{symm} and \eqref{sym1} gives
\begin{align*}
M'_R[u]
&= 4\int_{\mathbb{R}^N}\Big[\Big(\frac{\delta_{lk}}r
 -\frac{x_lx_k}{r^3}\Big)\partial_r\xi_R
 +\frac{x_lx_k}{r^2}\partial_r^2\xi_R\Big]
 \Re(\partial_ku\partial_l\bar u)\,dx
 -\int_{\mathbb{R}^N}\Delta^2\xi_R|u|^2\,dx\\
&\quad + 4\lambda\int_{\mathbb{R}^N}\partial_r\xi_R\frac{|u|^2}{|x|^3}\,dx
+2(\frac1q-1)\int_{\mathbb{R}^N}(\partial_r^2\xi_R+\frac{N-1}r\partial_r\xi_R)
 |x|^{-2\tau}|u|^{2q}\,dx\\
&\quad -\frac{4\tau}q\int_{\mathbb{R}^N}\frac{\partial_r\xi_R}r
 |x|^{-2\tau}|u|^{2q}\,dx\\
&= 4\int_{\mathbb{R}^N}\Big[\Big(\frac{|\nabla u|^2}r
 -\frac{|x\cdot\nabla u|^2}{r^3}\Big)\partial_r\xi_R
  +\frac{|x\cdot\nabla u|^2}{r^2}\partial_r^2\xi_R\Big]\,dx
  -\int_{\mathbb{R}^N}\Delta^2\xi_R|u|^2\,dx\\
&\quad + 4\lambda\int_{\mathbb{R}^N}\partial_r\xi_R\frac{|u|^2}{|x|^3}\,dx
 +2(\frac1q-1)\int_{\mathbb{R}^N}\Big(\partial_r^2\xi_R+({N-1}
 +\frac{2\tau}{q-1})\frac{\partial_r\xi_R}r\Big)
 |x|^{-2\tau}|u|^{2q}\,dx.
\end{align*}
Now, noting that $\lambda\geq0$, by \eqref{norm} and \eqref{prpr}, one obtains
\begin{align*}
M'_R[u]-8\mathcal J[u]
&= 4\int_{\mathbb{R}^N}\Big[\Big(\frac{|\nabla u|^2}r
 -\frac{|x\cdot\nabla u|^2}{r^3}\Big)\partial_r\xi_R
 +\frac{|x\cdot\nabla u|^2}{r^2}\partial_r^2\xi_R\Big]\,dx\\
&\quad -\int_{\mathbb{R}^N}\Delta^2\xi_R|u|^2\,dx+4\lambda\int_{\mathbb{R}^N}\partial_r
 \xi_R\frac{|u|^2}{|x|^3}\,dx-8\|\sqrt {\mathcal K_{\lambda}} u\|^2\\
&\quad -2\frac{q-1}q\int_{\mathbb{R}^N}\Big(\partial_r^2\xi_R+({N-1}
 +\frac{2\tau}{q-1})\frac{\partial_r\xi_R}r
 -2\frac{B'}{q-1}\Big)|x|^{-2\tau}|u|^{2q}\,dx\\
&\leq 4\int_{\mathbb{R}^N}\frac{|x\cdot\nabla u|^2}{r^2}\Big(\partial_r^2\xi_R
 -\frac{\partial_r\xi_R}r\Big)\,dx
 -\int_{\mathbb{R}^N}\Delta^2\xi_R|u|^2\,dx\\
&\quad -2\frac{q-1}q\int_{\mathbb{R}^N}\Big(\partial_r^2\xi_R+({N-1}
 +\frac{2\tau}{q-1})\frac{\partial_r\xi_R}r
 -2\frac{B'}{q-1}\Big)|x|^{-2\tau}|u|^{2q}\,dx.
\end{align*}
Using the estimate $\|\nabla^\gamma \xi_R\|_\infty\lesssim R^{2-|\gamma|}$
 via the mass conservation law, one has
\begin{equation}
\big|\int_{\mathbb{R}^N}\Delta^2\xi_R|u|^2\,dx\big|
\lesssim R^{-2}. \label{mr11}
\end{equation}
Moreover, one decomposes the above quantity as follows
\begin{equation} \label{Ai}
\begin{aligned}
M'_R[u]-8\mathcal J[u]
&\leq 4\int_{\mathbb{R}^N}\frac{|x\cdot\nabla u|^2}{r^2}\Big(\partial_r^2\xi_R
 -\frac{\partial_r\xi_R}r\Big)\,dx
-\int_{\mathbb{R}^N}\Delta^2\xi_R|u|^2\,dx \\
&\quad -2\frac{q-1}q\int_{\mathbb{R}^N}\Big(\partial_r^2\xi_R+({N-1}
 +\frac{2\tau}{q-1})\frac{\partial_r\xi_R}r-2\frac{B'}{q-1}
 \Big)|x|^{-2\tau}|u|^{2q}\,dx \\
&:= -(A_1)-2\frac{q-1}q\cdot(A_2).
\end{aligned}
\end{equation}
By the properties of $\xi_R$, 
$$
\partial_r^2\xi_R+({N-1}+\frac{2\tau}{q-1})\frac{\partial_r\xi_R}r
-2\frac{B'}{q-1}=0,\quad\text{for}\quad B(R).
$$
Thus, by the Gagliardo-Nirenberg estimate via the mass conservation law, 
\eqref{prpr} and \eqref{norm}, one obtains
\begin{equation} \label{3.241}
\begin{aligned}
(A_2)
&= \int_{B^c(R)}\Big(\partial_r^2\xi_R+({N-1}
 +\frac{2\tau}{q-1})\frac{\partial_r\xi_R}r-2\frac{B'}{q-1}\Big)
 |x|^{-2\tau}|u|^{2q}\,dx \\
&\lesssim R^{-2\tau}\int_{\mathbb{R}^N}|u|^{2q}\,dx \\
&\lesssim R^{-2\tau}\|\sqrt{\mathcal K_\lambda} u\|^{B'-2\tau}
 \|u\|^{2q-B'+2\tau} \\
&\lesssim R^{-2\tau}\|\sqrt{\mathcal K_\lambda} u\|^{B'-2\tau}.
\end{aligned}
\end{equation}
Since $\mathcal J[u]<0$, by the Gagliardo-Nirenberg estimate in
 Proposition \ref{gag2} via the mass conservation law, one has
\[
\|\sqrt{\mathcal K}_\lambda u\|^2
\lesssim \int_{\mathbb{R}^N}|x|^{-2\tau}|u|^{2q}\,dx
\lesssim \|\sqrt{\mathcal K_\lambda} u\|^{B'}\|u\|^{2q-B'}
\lesssim \|\sqrt{\mathcal K_\lambda} u\|^{B'}.
\]
Thus, $B'>2$ implies that there is $C_2>0$ such that for $t\in[0,T^*)$,
\begin{equation} \label{35.56}
\|\sqrt{\mathcal K_\lambda} u(t)\|\geq C_2.
\end{equation}
Thus, \eqref{Ai}-\eqref{3.241} and \eqref{nw'} give for $2<B'\leq2+2\tau$ 
and $R\gg 1$,
\begin{equation} \label{3.261}
\begin{aligned}
M'_R[u]
&\lesssim \mathcal J[u]+R^{-2}+R^{-2\tau}
 \|\sqrt{\mathcal K_\lambda} u\|^{B'-2\tau} \\
&\lesssim -\|\sqrt{\mathcal K_\lambda} u\|^2+R^{-2}+R^{-2\tau}
 \|\sqrt{\mathcal K_\lambda} u\|^{B'-2\tau} \\
&\lesssim \|\sqrt{\mathcal K_\lambda} u\|^2\Big(-1+R^{-2}+R^{-2\tau}
 \|\sqrt{\mathcal K_\lambda} u\|^{B'-2-2\tau}\Big) \\
&\lesssim -\|\sqrt{\mathcal K_\lambda} u\|^2.
\end{aligned}
\end{equation}
By time integration, \eqref{35.56}, and \eqref{3.261} imply that 
\begin{equation}\label{neg}
M_R[u(t)]\lesssim-t,\quad t>T>0.
\end{equation}
By time integration again, from \eqref{3.261} and \eqref{neg}, 
it follows that
\begin{equation}
M_R[u(t)]\lesssim-\int_T^t\|\sqrt{\mathcal K_\lambda} u(s)\|^2\,ds \label{3.271}.
\end{equation}
Now, the definition \eqref{mrwtz} via the mass
conservation law gives
\begin{equation} 
|M_R[u]|
=2|\Im\int_{\mathbb{R}^N} \bar u(\nabla\xi_R\cdot\nabla u)\,dx| 
\lesssim R\|\nabla u\|\|u\| 
\lesssim R\|\nabla u\|. \label{3.281}
\end{equation}
By \eqref{neg}, \eqref{3.271} and \eqref{3.281}, it follows that
\begin{equation}\label{3.291}
\int_T^t\|\sqrt{\mathcal K_\lambda} u(s)\|^2\,ds
\lesssim|M_R[u(t)]|\lesssim R\|\sqrt{\mathcal K_\lambda} u\|,\quad 
\forall t>T.
\end{equation}
Take the real function $f(t):=\int_T^t\|\sqrt{\mathcal K_\lambda} u(s)\|^2$. 
By \eqref{3.291}, one obtains $f^2\lesssim f'$. Like previously, this ODI
has no global solution. This completes the proof.

\subsubsection*{Second case}
The proof is similar to the previous section.

\subsection{Bi-harmonic case}
In this sub-section, one assumes that $s=2$.

\subsubsection*{First case }
Assume that \eqref{ss'} holds. We start with the next auxiliary result 
which can be proved arguing as in Lemma \ref{2ky}.

\begin{lemma}\label{2ky'}
 \begin{enumerate}
 \item The set $\mathcal{A'}^-$ is stable under the flow of \eqref{S'}.
 \item There exists $\varepsilon>0$, such that for any $t\in [0, T^*)$,
 \begin{equation}
 \mathcal J[u(t)]+\varepsilon \|{\Delta}u(t)\|^2
 \leq-\frac{B'}4\big(m'-\mathcal{S'}[u(t)]\big).\label{2nw'}
 \end{equation}
 \end{enumerate}
\end{lemma}

Using the estimates \eqref{mr1} and \eqref{mr2} via Morawetz identity 
\eqref{mrw2}, one obtains
\begin{equation}
\begin{aligned}
-M_R'
&=2\int_{\mathbb{R}^N}\Big(2\partial_{jk}\Delta\xi_R\partial_ju\partial_k\bar u
 -\frac12(\Delta^3\xi_R)|u|^2-4\partial_{jk}\xi_R\partial_{ik}u\partial_{ij}
 \bar u \\
&\quad+\Delta^2\xi_R|\nabla u|^2+\frac{q-1}{q}(\Delta\xi_R)|x|^{-2\tau}
 |u|^{2q}-\frac1q\nabla\xi_R\cdot\nabla(|x|^{-2\tau})|u|^{2q}\Big)\,dx \\
&=\frac{8B'}{q}\int_{B(R)}|x|^{-2\tau}|u|^{2q}\,dx-8\int_{\mathbb{R}^N}
 \partial_{jk}\xi_R\partial_{ik}u\partial_{ij}\bar u\,dx+O(R^{-4})
 +\|\nabla u\|^2O(R^{-2}) \\
&\quad +2\frac{q-1}{q}\int_{B^c(R)}(\Delta\xi_R)|x|^{-2\tau}|u|^{2q}\,dx
 -\frac2q\int_{B^c(R)}\nabla\xi_R\cdot\nabla(|x|^{-2\tau})|u|^{2q}\,dx. 
\end{aligned}
\end{equation}
Thus, by \eqref{mr3}, one writes
\begin{align*}
-M_R'
&=\frac{8B'}{q}\int_{B(R)}|x|^{-2\tau}|u|^{2q}\,dx
 -\frac2q\int_{B^c(R)}\nabla\xi_R\cdot\nabla(|x|^{-2\tau})|u|^{2q}\,dx \\
&\quad +2\frac{q-1}{q}\int_{B^c(R)}(\Delta\xi_R)|x|^{-2\tau}|u|^{2q}\,dx
 +O(R^{-4})+\|\nabla u\|^2O(R^{-2}) \\
&\quad -8\sum_{i=1}^N\int_{\mathbb{R}^N}|\nabla u_i|^2\frac{\partial_r\xi_R}{|x|}\,dx
 -8\sum_{i=1}^N\int_{\mathbb{R}^N}\frac{|x\cdot\nabla u_i|^2}{|x|^2}
 \Big(\partial_r^2\xi_R-\frac{\partial_r\xi_R}{|x|}\Big)\,dx \\
&=-16\mathcal J[u]-{\frac{8B'}{q}\int_{B^c(R)}|x|^{-2\tau}|u|^{2q}\,dx}
 +O(R^{-4})+\|\nabla u\|^2O(R^{-2}) \\
&\quad -\frac2q\int_{B^c(R)}\nabla\xi_R\cdot\nabla(|x|^{-2\tau})|u|^{2q}\,dx
+2\frac{q-1}{q}\int_{B^c(R)}(\Delta\xi_R)|x|^{-2\tau}|u|^{2q}\,dx \\
&\quad -8\Big(\sum_{i=1}^N\int_{\mathbb{R}^N}|\nabla u_i|^2
 \big(\frac{\partial_r\xi_R}{|x|}-2\big)\,dx
 +\sum_{i=1}^N\int_{\mathbb{R}^N}\frac{|x\cdot\nabla u_i|^2}{|x|^2}
 \big(\partial_r^2\xi_R-\frac{\partial_r\xi_R}{|x|}\big)\,dx\Big) .
\end{align*}
Then, by an interpolation argument and Young estimate, \eqref{prpr} gives
\begin{equation}
 M_R'
\lesssim\mathcal J[u]+O\Big(\int_{B^c(R)}|x|^{-2\tau}|u|^{2q}\,dx\Big)
+R^{-2}+R^{-2}\|\Delta u\|^2 .
\end{equation}
Since $1<q<\frac N{N-4}$, by the Gagliardo-Nirenberg inequality, one writes
\[
\int_{B^c(R)}|x|^{-2\tau}|u|^{2q}\,dx
\leq cR^{-2\tau}\|u\|_{2q}^{2q}
\leq cR^{-2\tau}\|\Delta u\|^{\frac N2(q-1)}.
\]
So, 
\begin{equation}
M_R'[u]
\lesssim\mathcal J[u]+R^{-2\tau}\|\Delta u\|^{B'-\tau}+R^{-2}. \label{4.24}
\end{equation}
Since $\mathcal J[u]<0$, by the Gagliardo-Nirenberg estimate in 
Proposition \ref{gag2} via the mass conservation law, one has
\[
\|\Delta u\|^2
\lesssim\mathcal Q[u]
\lesssim\|\Delta u\|^{B'}\|u\|^{2q-B'}
\lesssim\|\Delta u\|^{B'}.
\]
Thus, $B'>2$ implies that there is $C_3>0$ such that for any $t\in[0,T^*)$,
\begin{equation} \label{3744}
\|\Delta u(t)\|\geq C_3.
\end{equation}
Thus, \eqref{2nw'}-\eqref{3744} give for $2<B'\leq2+\tau$ and 
$R\gg 1$,
\begin{equation} \label{4.26}
\begin{aligned}
 M_R'[u]
&\lesssim \mathcal J[u]+R^{-2}+R^{-2}\|\Delta u\|^2+R^{-2\tau}
 \|\Delta u\|^{B'-\tau} \\
&\lesssim -\|\Delta u\|^2+R^{-2}+R^{-2}\|\Delta u\|^2
 +R^{-2\tau}\|\Delta u\|^{B'-\tau} \\
&\lesssim \|\Delta u\|^2\Big(-1+R^{-2}+R^{-2\tau}
 \|\Delta u\|^{B'-2-\tau}\Big) \\
&\lesssim -\|\Delta u\|^2.
\end{aligned}
\end{equation}
By time integration, \eqref{3744} and \eqref{4.26} imply that 
\begin{equation} \label{neg'}
M_R[u(t)]\lesssim-t,\quad t>T>0 
\end{equation}
By time integration again, from \eqref{4.26} and \eqref{neg'}, 
it follows that
\begin{equation}
 M_R[u(t)]
\lesssim -\int_T^t\|\Delta u(s)\|^2\,ds\label{4.27}.
\end{equation}
The rest of the proof follows as previously.

\subsection*{Acknowledgements} 
 R. Bai was supported by the Postdoctoral Fellowship Program of CPSF 
 (Grant No. GZC20230694). 
 The authors want to thank the anonymous referees for the 
 helpful comments and suggestions.

\end{document}